\documentclass[sigconf, screen=true, review=false]{acmart}

\usepackage{mathtools}
\usepackage{cleveref}
\usepackage{paralist}
\usepackage{etoolbox}
\newbool{tabularray}
\booltrue{tabularray}
\ifbool{tabularray}{
\usepackage{tabularray}}{}
\usepackage[disable,colorinlistoftodos,bordercolor=orange,backgroundcolor=orange!20,linecolor=orange,textsize=scriptsize]{todonotes}
\usepackage{bbm}
\usepackage{array}
\usepackage{tabularx}

\usetikzlibrary{calc}
\usetikzlibrary{shapes.geometric}
\usetikzlibrary{automata, positioning}
\usetikzlibrary{arrows}
\usetikzlibrary{arrows.meta}
\usepackage{tikz-3dplot}

\usepackage{blkarray} %

\newcommand{\R}{\mathbb{R}}

\newcommand{\N}{\mathbb{N}}

\newcommand{\places}{\mathcal{P}}

\newcommand{\ie}{i.e.}
\newcommand{\eg}{e.g.\ }

\newcommand{\transitions}{\mathcal{Q}}

\newcommand{\inc}{^{\textrm{in}}}

\newcommand{\cL}{\mathcal{L}}
\newcommand{\Laurent}{\cL_{-1}}
\newcommand{\eps}{\varepsilon}
\newtheorem{assumption}{Assumption}

\newtheorem{assumptionprime}{Assumption}

\newtheorem{theorem}{Theorem}
\newtheorem{corollary}[theorem]{Corollary}
\newtheorem{lemma}[theorem]{Lemma}
\newtheorem{proposition}[theorem]{Proposition}
\theoremstyle{remark}

\newtheorem{remark}[theorem]{Remark}

\newcommand{\actions}{\mathcal{A}}
\newcommand{\timedelays}{\mathcal{T}}
\newcommand{\lc}{c}
\newcommand{\LC}{C}

\definecolor{tablegray}{gray}{0.94}
\definecolor{headergray}{gray}{0.82}

\AtBeginDocument{%
  \providecommand\BibTeX{{%
    \normalfont B\kern-0.5em{\scshape i\kern-0.25em b}\kern-0.8em\TeX}}}

\setcopyright{acmcopyright}
\copyrightyear{2021}
\acmYear{2021}
\acmDOI{10.1145/1122445.1122456}

\acmConference[]{}{}{}
 \acmBooktitle{}
 \acmPrice{}
 \acmISBN{}

\begin{document}
\nocite{Koh80}
\newcommand\myshade{85}
\definecolor{hreflinkcolor}{HTML}{3f87d3}
\definecolor{hrefcitecolor}{HTML}{ee7c27}
\definecolor{hrefurlcolor}{rgb}{1,0.5,0}
\hypersetup{
  linkcolor  = hreflinkcolor!85!black,
  citecolor  = hrefcitecolor!85!black,
  urlcolor   = hrefurlcolor!85!black,
  colorlinks = true,
}

\title{Stationary regimes of piecewise linear dynamical systems with priorities}

\author{Xavier Allamigeon}%
 \affiliation{%
   \institution{Inria Saclay \\ CMAP, \'Ecole polytechnique, IP Paris}
   \country{France}
 }
 \email{xavier.allamigeon@inria.fr}
 \author{Pascal Capetillo}%
 \affiliation{%
   \institution{Inria Saclay \\ CMAP, \'Ecole polytechnique, IP Paris}
   \country{France}
 }
 \email{pascal.capetillo@inria.fr}

 \author{St\'ephane Gaubert}%
 \affiliation{%
   \institution{Inria Saclay \\ CMAP, \'Ecole polytechnique, IP Paris}
   \country{France}
 }
 \email{stephane.gaubert@inria.fr}

\renewcommand{\shortauthors}{Allamigeon, Capetillo and Gaubert}
\begin{abstract}
Dynamical systems governed by priority rules appear in the modeling of emergency organizations and road traffic. These systems can be modeled by piecewise linear time-delay dynamics, specifically using Petri nets with priority rules. A central question is to show the existence of stationary regimes (\ie, steady state solutions)---taking the form of invariant half-lines---from which essential performance indicators like the throughput and congestion phases can be derived. Our primary result proves the existence of stationary solutions under structural conditions involving the spectrum of the linear parts within the piecewise linear dynamics. This extends to a broader class of systems a fundamental theorem of Kohlberg (1980) dealing with nonexpansive dynamics. The proof of our result relies on topological degree theory and the notion of ``Blackwell optimality'' from the theory of Markov decision processes. Finally, we validate our findings by demonstrating that these structural conditions hold for a wide range of dynamics, especially those stemming from Petri nets with priority rules. This is illustrated on real-world examples from road traffic management and emergency call center operations.
\end{abstract}


\maketitle

\section{Introduction}

\paragraph{Context}
A series of works, based on the ``max-plus approach'' to discrete event systems and Petri nets~\cite{baccelli1992sync,HOW:05}, have established how systems with synchronization and concurrency phenomena can be represented by piecewise linear dynamics.
The theory has been thoroughly developed for {\em monotone\/} (\ie, order-preserving) piecewise linear dynamics. Under natural assumptions, the existence of stationary solutions (\ie, steady state solutions) has been established~\cite{CGQ95b,baccelli1996free,gaujal2004optimal} and it has been shown that all trajectories remain within a bounded neighborhood of a stationary solution, see~\cite{boyet2021}. This provides a well-defined notion of the throughput---a fundamental tool in the analysis of these systems.

In contrast, the case of {\em nonmonotone\/} piecewise linear dynamics is more challenging. Nonmonotonicity naturally appears in systems with {\em priority rules}, as seen in applications to road traffic~\cite{Farhi2011} and emergency call centers~\cite{emergency15,boyet2021,boyet2022}. In these cases, stationary solutions have been derived on a case-by-case basis, for instance to compute {\em congestion phases\/}---which combination of events and resources that limit processing speeds---in terms of the resource allocation~\cite{emergency15,boeuf2017dynamics,boyet2022}.

To date, there is no general result that ensures the existence of stationary solutions for piecewise linear systems beyond the monotone case, which was posed as an open question in~\cite{maxplusblondel}. This raises an intriguing problem which we address in this paper: proving, {\em a priori}, the existence of stationary solutions of nonmonotone piecewise linear dynamics.  

\paragraph{Setting} We consider piecewise linear dynamical systems with time delays, whose trajectories are vector-valued functions of time, $x(t)= (x_i(t))_{i\in[n]} \in \R^n$ satisfying the following conditions:
\begin{align}\label{eq:tds}
  x_i(t) = \min_{a\in \actions_i} \Bigg( r_i^a + \sum_{\tau\in \timedelays} \sum_{j\in [n]} {(P_\tau^a)}_{ij} x_j(t-\tau ) \Bigg), \quad i\in [n], \; t\geq 0 \,. \tag{\(D\)} 
\end{align}
Here, $\timedelays$ is a finite subset of $\N \coloneqq \{0,1,2,\dots\}$, \(\actions_i\) is a finite set for each \(i \in [n]\), and \(r_i^a\) as well as \((P^a_\tau)_{ij}\) are real numbers (of any sign), for all $a \in \actions_i$ and $j \in [n]$. In the context of Petri net models dealt with in~\cite{baccelli1992sync,HOW:05,CGQ95b,gaujal2004optimal,Farhi2011,emergency15,boyet2021,boyet2022}, the functions $x_i(\cdot)$ are known as {\em counter functions}, \ie, $x_i(t)$ is the number of events of type $i\in [n]$ that have occurred up to time $t$. The system is said to be {\em monotone\/} when the trajectory $x(\cdot)$ is a nondecreasing function of the initial condition $(x(s))_{-\bar{\tau}\leq s\leq 0}$ where $\bar{\tau}\coloneqq \max \timedelays$.
Monotonicity holds when all the \((P^a_\tau)_{ij}\) are nonnegative. Note that in the cases where $\timedelays$ is reduced to $\{1\}$ (\ie, $\tau \equiv 1$) and, for every $i$ and $a \in \actions_i$, the sum of the \((P^a_\tau)_{ij}\) for $j \in [n]$ is equal to $1$, \eqref{eq:tds} encompasses the dynamics of Markov decision processes (MDPs)~\cite{puterman}. More generally, when the $(P^a_\tau)_{ij}$ are nonnegative and $\sum_{\tau \in\timedelays} \sum_{j\in[n]} (P^a_\tau)_{ij}=1$ for all $i$ and $a \in \actions_i$, ~\eqref{eq:tds} coincides with the dynamic programming equation of a {\em semi-Markov\/} decision process~\cite[Chap.~11]{puterman};~\cite{CGQ95b} and~\cite[Chap. 4]{boyet2022}. In contrast, systems with negative values for \((P^a_\tau)_{ij}\) emerge naturally when studying Petri nets with priorities. An example requiring this general setting is recalled in~\Cref{sec:motivation}.

\paragraph{Contribution}

In this paper, we establish structural conditions under which the existence of stationary solutions of nonmonotone piecewise dynamical systems of the form~\eqref{eq:tds} is guaranteed (\Cref{cor:main}). The stationary solutions take the form of an invariant half-line $x(t) = u + \rho t$, with $u,\rho\in\R^n$, where $\rho$ corresponds to the throughput of system. More specifically, the structural conditions are expressed in terms of spectral properties of the linear components of the dynamics, \ie, the $(P^a_\tau)_{ij}$. Remarkably, these conditions are independent of the parameters $r_i^a$. This is of interest, for instance, when studying system performance (\eg the throughput) as the parameters vary. In practical applications, the parameters~$r_i^a$ may correspond to resources such as staff allocation: a key consideration for decision makers in optimizing such organizations, see~\cite{emergency15,boyet2021}. 

Additionally, we show that all but one of these spectral conditions are automatically fulfilled by Petri nets with priority rules satisfying a stoichiometry assumption (a kind of conservation law which is always satisfied in our applications).

We demonstrate the applicability of our results by recovering the existence of stationary solutions for a range of examples of Petri nets arising from real-world applications in road traffic and emergency call centers.

The proof of \Cref{cor:main} employs techniques from topological degree theory and discounted MDPs. Indeed, we introduce an auxiliary parameter $\alpha\in [0,1)$, which may be thought of as a discount factor, and use degree theory (i.e., homotopy arguments) to establish the existence of a fixed point to the discounted problem for all values of $\alpha$. Unlike in standard MDPs, the resulting fixed-point mapping in this discounted problem is generally non-contracting. However, we demonstrate that the typical contraction condition can be replaced with a semisimplicity condition on the spectrum of specific matrices. Finally, we derive the existence of a steady state from the existence of an appropriate Laurent series expansion of the discounted solution.

In addition, we show in \Cref{thm:maxmin} and \Cref{cor:kohl} that our result extends (relaxing the nonexpansiveness condition) a fundamental theorem due to Kohlberg~\cite{Koh80}, showing the existence of an invariant half-line for piecewise linear nonexpansive dynamics with unit delays.

\paragraph{Related work}
As mentioned above, the open problem of the existence of stationary solutions of piecewise linear dynamics was originally stated by Plus in~\cite{maxplusblondel}. The existence of such stationary solutions has been established in specific applications, such as by Farhi, Goursat, and Quadrat for road networks with priorities~\cite{Farhi2011,nadir,farhi2019dynamicprogrammingsystemsmodeling} and by Allamigeon, B\oe uf, Boyet, and Gaubert for emergency call centers~\cite{emergency15,boeuf2017dynamics,boyet2021,boyet2022}. This was carried out by an explicit enumeration (done by hand) of the stationary solutions of the (exponentially many) linear parts of the dynamics. The conditions under which such a stationary solution applies to the whole dynamical system writes as linear inequality constraints over the parameters $r^a_i$. Such constraints were then solved, again, by hand. This approach can be automatized but remains limited to small systems. In contrast, our main result is motivated by identifying broad classes of dynamical systems, encompassing the mentioned examples, for which the existence of stationary solutions is guaranteed {\em a priori}, especially independently of the parameters $r^a_i$.

The present results are also inspired by the recent work of Radons and Tonneli-Cueto~\cite{Radons2023}, showing the existence of solutions of the ``absolute value equation'' arising in the solution of linear complementarity problems, using topological degree theory methods. Their results apply to {\em fixed points\/} of a class of piecewise linear maps, that are {\em positively homogeneous\/} of degree $1$. In contrast, our main results concern the more delicate notion of invariant half-lines, and dynamics that are not positively homogeneous. Moreover, an essential new ingredient of our approach is the reduction to the discounted problem, extending the notion of {\em Blackwell optimality\/} from the theory of MDPs~\cite{puterman} to the case of maps that are no longer nonexpansive and order-preserving. 
Hence, the present results provide a partial extension of stochastic dynamic programming to the case of ``negative probabilities'' --- a concept originally considered by Feynman~\cite{feynman}. 

\paragraph{Outline} In \S\ref{sec:motivation} we present the framework giving rise to the dynamics~\eqref{eq:tds}, namely that of Petri nets with priority and preselection rules. We also introduce a running example of these dynamics which arises from a simplified model of an emergency call center operating in the Paris area. %
In \S\ref{sec:main}, we formulate the spectral conditions on the linear parts of the dynamics, namely Assumptions~\ref{as:1} and~\ref{as:2}, and we prove our main result, \Cref{cor:main}, on the existence of stationary solutions. An extension of this result to more general piecewise linear dynamics involving both ``min'' and ``max'' is carried out in \S\ref{sec:kohl}, which allows us to recover Kohlberg's theorem as a special case. The link back to Petri nets is made in \S\ref{sec:automaticPetri} by showing that all but one of the assumptions presented in \S\ref{sec:main} are automatically fulfilled by the broad class of dynamics introduced in \S\ref{sec:motivation}. Finally in \S\ref{sec:catalog}, we apply our main results to a set of examples from real-world applications.

\section{Motivation}\label{sec:motivation}

\subsection{Petri nets with Priority Rules}

Our approach relies on models of timed Petri nets with priority considerations, which we next present, referring the reader to~\cite{Farhi2011,emergency15,boyet2021,boyet2022} for a complete background.

A {\em Petri net\/} is a bipartite graph consisting of places, \(\places\), and transitions, \(\transitions\), together with a vector of markings \(m \in \N^{\places}\) associating a natural number \(m_p\) of {\em tokens} to each place \(p \in \places\). Its dynamics builds on transition {\em firings}. Firing a transition consumes one token in each upstream place and creates a new token in each downstream place. A transition which has available tokens in all its upstream places is called {\em fireable}.
We enrich the notion of a Petri net by introducing a collection \(\timedelays = (\tau_p)_p \in \N^\places\)
of {\em holding times}: for each place \(p\), $\tau_p$ represents the amount of time a token must spend in this place before becoming available for firing. Once a token has spent \(\tau_p\) time in place \(p\) it is called {\em mature}. 

We also allow for {\em input transitions}, transitions with no upstream place and externally prescribed firings producing tokens in their downstream places. This puts us in the framework of {\em controlled\/} Petri nets.

Determining the order in which transitions fire is not always straight-forward as there may exist {\em structural conflict}, the case of a place having several downstream transitions, in the system.
In this case some routing policy must be adopted to determine which fireable transition to allocate a mature token to. In order to model the complex concurrency phenomena present in emergency systems, road traffic, and further situations, two such conflict resolution structures are presented: {\em preselection rules\/} and {\em priority rules}.

A preselection rule is defined by a discrete probability distribution \(\pi\) over the downstream transitions of a place \(p\). Adding the constraint that each such downstream transition \(q\) can have only one upstream place, this resolves the structural conflict by allocating any mature token in \(p\) to \(q\) with probability \(\pi_q^p\).

A priority rule associated to a place \(p\) is defined by a total order \(\prec_p\) on the downstream transitions \(q\). In this case, the structural conflict is resolved by prioritizing the mature tokens to go to the transition with highest priority as long as it is fireable. Specifically, if two downstream transitions are fireable, the token is used by the one with higher priority. One further assumption that is placed on the priority rules is that they are {\em consistent\/} meaning that there is a linear extension of the total orders to all transitions that is acyclic as a topological ordering. In other words, there exists some (not necessarily unique) total ordering of all transitions that is consistent with the priority rules. This becomes a way to resolve structural conflict for transitions with several upstream places.

To express the dynamics of the system, we rely on {\em counter functions}. We associate to every transition a nondecreasing function counting the number of times it has been fired up to (and including) a given time.
In general, the dynamics of a Petri net can be described by a set of equations
describing the evolution of \(z_q(t) \colon \R \to \N\), where only a subset of transitions \(q\) are considered. In order to model priority rules a ``smallest time-delay'' \(\eps\) must be added to the set of time-delays \(\timedelays\) to allow counting how many times a transition has been fired until and not including a certain time \(t\). In exact models of Petri nets, counter functions are integer valued. A continuous relaxation is considered in~\cite{Farhi2011,emergency15,boyet2021,boyet2022}, allowing ``infinitesimal'' firings, leading to counters that are real valued. Note that, in this setting, pre-selection rules define what fraction of a mature token goes to what downstream transition. It is shown in~\cite[Chap. 2, Sec. 3]{boyet2022} that, at least for a subclass of Petri nets, the continuous relation is asymptotically tight when a scaling limit is considered (for instance, it is accurate for a large scale call center, with many operators). Then, after a continuous relaxation, the counter equations of a timed Petri net with priorities satisfy a system of the form \eqref{eq:tds}. We refer to~\cite{emergency15,boyet2022} for the detailed derivation of the equations~\eqref{eq:tds}  for general Petri nets, and to \S\ref{sec:automaticPetri} below for details.

\subsection{Paris Emergency Call Center as a Running Example}

We shall consider the model of an emergency call center
developed in~\cite{emergency15}: PFAU ({\em Plateforme d'appels  d'urgence\/}) is the call center handling calls to the firefighters and policemen (emergency numbers 17-18-112) in the Paris area.

We assume a constant rate of incoming calls, denoted as \(\lambda\). All calls are initially answered by a first level responder when available and are subsequently sorted into three degrees of severity: non-urgent, urgent, and very urgent (potentially life threatening situation), based on a preliminary assessment made by the responder. This is modeled with a preselection rule where the probability distribution assigns probabilities (or relative rates) to each of the different severities occurring in an incoming call. Non-urgent callers are fully handled by the first level responder. Urgent callers require assistance from a second level operator and are transferred to a waiting room until a second level responder is available. Very urgent callers require a three-way briefing of the second level (specialized) responder by the first level responder, to ensure the caller is constantly monitored and essential information is properly transferred. In particular, such very urgent calls wait for an available second level responder while staying in line with a first level responder. Once the three-way conversation concludes, the first level responder returns to the pool while the rest of the call is handled exclusively by the second level responder. Second level responders prioritize very urgent calls over urgent calls, which is modeled with a priority rule. In the beginning there are \(N_A\) first level responders and \(N_P\) second level responders available. A working example of a Petri net with preselection and priority rules modeling this situation, taken from~\cite{emergency15}, is presented in~\Cref{fig:bilevel}. Note the convention of notation: the double arrow upstream transition $z'_3$, representing the start of a three way conversation, indicates that this task has priority over transition $z'_3$ (handling of a urgent call) -- with a single upstream arrow.

\begin{figure}[h]
\centering
\def\tkzscl{.27}
\vspace{-.2cm}

\definecolor{colorARM}{rgb}{0,0,1}
\definecolor{colorARMres}{rgb}{0.92,0.5,0.11}
\definecolor{colorAMU}{rgb}{1,0,0}
\definecolor{colorexit}{rgb}{0.4,0.4,0.4}

\tikzset{place/.style={draw,circle,inner sep=2.5pt,semithick}}
\tikzset{transition/.style={rectangle, thick,fill=black, minimum width=2mm,inner ysep=0.5pt, minimum width=2mm}}
\tikzset{jeton/.style={draw,circle,fill=black!80,inner sep=.35pt}}
\tikzset{pre/.style={=stealth'}}
\tikzset{post/.style={->,shorten >=1pt,>=stealth'}}
\tikzset{-|/.style={to path={-| (\tikztotarget)}}, |-/.style={to path={|- (\tikztotarget)}}}
\tikzset{Farhi/.style 2 args={dashed,dash pattern=on 1pt off 1pt,#1, postaction={draw,dashed,dash pattern=on 1pt off 1pt,#2,dash phase=1pt}}}
\tikzset{arrowPetri/.style={>=latex,rounded corners=5pt,semithick}}

\begin{tikzpicture}[scale=\tkzscl,font=\scriptsize]

\def\p{2.2}

\begin{scope}[shift={(0,4)}]
\node[place] (pool_first_level) at ($(-2*\p,-1.5*\p)$) {};
\node (txt_Na) at ($(pool_first_level)+(-1.3,0)$) {$N_A$};

\node[transition]    (q_arrivals)                          at    (0, 0) {};
\node[place]         (p_inc_calls)                         at    ($(q_arrivals) + (0, \p)$)  {};
\node[transition]    (q_inc_calls)                         at    ($(p_inc_calls) + (0, \p)$)  {};
\node[place]         (p_arrivals)       at    ($(q_arrivals) + (0,-\p)$) {};
\node[transition]    (q_begin_NU)       at    ($(p_arrivals) + (0,-\p)$) {};
\node[transition]    (q_begin_U)        at    ($(p_arrivals) + (2*\p,-\p)$) {};
\node[transition]    (q_begin_VU)       at    ($(p_arrivals) + (4*\p,-\p)$) {};

\draw[->,arrowPetri,colorARM] (p_arrivals)  |- ($(p_arrivals)+(.5*\p,-.5*\p)$) -| (q_begin_U);
\draw[->,arrowPetri,colorARM] (q_begin_NU) |- ($(q_begin_NU)+(0,-.5*\p)$) -| (pool_first_level);
\draw[->,arrowPetri,colorARM] (q_begin_U) |- ($(q_begin_NU)+(0,-.5*\p)$) -| (pool_first_level);
\draw[->,arrowPetri,colorARM] (p_arrivals)  -- (q_begin_NU);
\draw[->,arrowPetri,colorARM] (q_arrivals) -- (p_arrivals);
\draw[->,arrowPetri,colorARM] (p_arrivals)  |- ($(p_arrivals)+(.5*\p,-.5*\p)$) -| (q_begin_VU);

\node[place]         (p_U)           at    ($(q_begin_U)  + (0,-\p)$) {};
\node[transition]    (q_U)         at    ($(p_U)  + (0,-\p)$) {};
\node[place]         (p_VU1)           at    ($(q_begin_VU)  + (0,-\p)$) {};
\node[transition]    (q_VU1)         at    ($(p_VU1)  + (0,-\p)$) {};
\node[place]         (p_VU2)           at    ($(q_VU1)  + (0,-\p)$) {};
\node[transition]    (q_VU2)         at    ($(p_VU2)  + (0,-\p)$) {};

\draw[arrowPetri]    (q_inc_calls) |- ($(q_inc_calls)+(-.35*\p,.25*\p)$);
\draw[dashed, arrowPetri]    (q_inc_calls) |- ($(q_inc_calls)+(-1.5*\p,.25*\p)$);
\draw[->,arrowPetri] (q_inc_calls) -- (p_inc_calls);
\draw[->,arrowPetri] (p_inc_calls) -- (q_arrivals);
\draw[->,arrowPetri,colorARM] (q_VU2) -- ($(q_VU2)+(0,-.5*\p)$)  -|  (pool_first_level);

\node[place]    (p_consult_AMU)    at    ($(q_VU2)+(-\p,-1.5*\p)$) {};

\draw[->,arrowPetri]    (q_begin_U) -- (p_U);
\draw[->,arrowPetri]    (p_U) -- (q_U);
\draw[->,arrowPetri,colorARM]    (q_begin_VU) -- (p_VU1);
\draw[->,arrowPetri,colorARM]    (p_VU1) -- (q_VU1);
\draw[->,arrowPetri,Farhi={colorARM}{colorAMU}]    (q_VU1) -- (p_VU2);
\draw[->,arrowPetri,Farhi={colorARM}{colorAMU}]    (p_VU2) -- (q_VU2);

\draw[->,arrowPetri,colorARM]                           (pool_first_level)    |- ($(q_arrivals)+(-.5*\p,.5*\p)$) -- (q_arrivals);

\node (txt_tau2) at ($(p_VU2)+(.5*\p,0)$) {$\tau_2$};
\node (txt_tau1) at ($(p_arrivals.center)+(.5*\p,0)$) {${\tau}_1$};
\node (txt_tau3) at ($(p_consult_AMU.center)+(.5*\p,0)$) {${\tau}_3$};
\node (txt_piNU) at ($(q_begin_NU.west)$) [left] {$\pi_{\textrm{NU}}$};
\node (txt_piU) at ($(q_begin_U.west)$) [left] {$\pi_{\textrm{U}}$};
\node (txt_piVU) at ($(q_begin_VU.west)$) [left] {$\pi_{\textrm{VU}}$};

\node[place]      (pool_second_level)    at    ($(q_VU2)+(2*\p,.5*\p)$) {};
\node (txt_Nm) at ($(pool_second_level)+(-1.3,0)$) {$N_P$};

\node[transition] (q_end_consult_AMU)                  at    ($(p_consult_AMU)+(0,-\p)$) {};
\draw[->,arrowPetri,colorAMU]    (q_U) |-  ($(p_consult_AMU) + (-0.5*\p, 0.5*\p)$)     -- (p_consult_AMU);
\draw[->,arrowPetri,colorAMU]    (q_VU2) |-  ($(p_consult_AMU) + (0.5*\p, 0.5*\p)$)     -- (p_consult_AMU);
\draw[->,arrowPetri,colorAMU]    (p_consult_AMU)     -- (q_end_consult_AMU);
\draw[->,arrowPetri,colorAMU]    (q_end_consult_AMU) -- ($(q_end_consult_AMU)+(0,-1)$) -| (pool_second_level);
\draw[->,arrowPetri,colorAMU]    (pool_second_level)      |- ($(q_U)+(.5*\p,.5*\p)$)      -- (q_U);
\draw[->>,arrowPetri,colorAMU]    (pool_second_level)      |- ($(q_VU1)+(.5*\p,.5*\p)$)      -- (q_VU1);

\node (txt_z0) at ($(q_inc_calls)+(1*\p,0)$) {$z_0 = \lambda t$};
\node (txt_z1) at ($(q_arrivals)+(.5*\p,0)$) {$z_1$};
\node (txt_z2) at ($(q_begin_NU)+(.5*\p,0)$) {$z_2$};
\node (txt_z2p) at ($(q_begin_U)+(.5*\p,0)$) {$z_2^\prime$};
\node (txt_z2pp) at ($(q_begin_VU)+(.5*\p,0)$) {$z_2^{\prime\prime}$};
\node (txt_z3) at ($(q_U)+(.5*\p,0)$) {$z_3$};
\node (txt_z3p) at ($(q_VU1)+(.5*\p,0)$) {$z_3^\prime$};
\node (txt_z4) at ($(q_VU2)+(.5*\p,0)$) {$z_4$};
\node (txt_z5) at ($(q_end_consult_AMU)+(.5*\p,0)$) {$z_5$};
\end{scope}
\end{tikzpicture}	
\caption{Simplified model of the ``17-18-112 PFAU'' emergency call center~\cite{emergency15}}
\vspace{-.3cm}
\label{fig:bilevel}
\end{figure}

As shown in~\cite{emergency15}, the dynamics of this system can, after elimination of trivial equations, be expressed in the essential variables $z_1,z_3,z'_3$:
\begin{subequations}
  \begin{align}
  z_1(t) & =  \min\left\{ z_0(t), N_{A} + (1-\pi_{\textrm{VU}})z_1(t-\tau_1) + z_3^\prime(t-\tau_2)\right\}\\
  z_3(t) & =  \min\left\{ \pi_{\textrm{U}}z_1(t-\tau_1), \right. \nonumber\\
         &   \phantom{\min\{} \left. N_{P} + z_3(t-\tau_2) + z_3^\prime(t-\tau_2-\tau_3) - z_3^\prime(t)\right\}\\
  z_3^\prime(t) & =  \min\left\{ \pi_{\textrm{VU}}z_1(t-\tau_1), \right. \nonumber\\
               &    \phantom{\min\{} \left. N_{P} + z_3(t-\tau_2) + z_3^\prime(t-\tau_2-\tau_3) - z_3(t^-)\right\}
  \end{align}
  \label{eq:bilevel_simple}
\end{subequations}
It was shown in~\cite{emergency15} that for all amounts of initial resources, $N_A,N_P$, there is a steady state solution. The throughput vector, \((\rho_1, \rho_3, \rho_3^\prime)\), is given by

\begin{align*}
  \rho_1 &= \min \left\{ \lambda, \frac{N_A}{\tau_1 + \pi_{\textrm{VU}} \tau_2}, \frac{N_P}{\pi_{\textrm{VU}}(\tau_2 + \tau_3)} \right\}, \\
  \rho_3 &= 
  \begin{cases}
    \pi_{\textrm{U}} \rho_1, & \text{if } N_P \geq r \min \left(N_A, \lambda(\tau_1 + \pi_{\textrm{VU}} \tau_2) \right) \\[10pt]
    \frac{N_P - \pi_{\textrm{U}} \rho_1 (\tau_2 + \tau_3)}{\tau_2}, & \text{if } r' \min \left(N_A, \lambda(\tau_1 + \pi_{\textrm{VU}} \tau_2)\right) \leq N_P \\
    & \quad \leq r \min \left(N_A, \lambda(\tau_1 + \pi_{\textrm{VU}} \tau_2)\right), \\[10pt]
    0, & \text{if } N_P \leq r' \min \left(N_A, \lambda(\tau_1 + \pi_{\textrm{VU}} \tau_2)\right)
  \end{cases} \\
  \rho_3' &= \pi_{\textrm{VU}} \rho_1
\end{align*}
\begin{align*}
    \text{where} \quad r &= \frac{\pi_{\textrm{VU}}\tau_2 + (\pi_{\textrm{U}} + \pi_{\textrm{VU}})\tau_3}{\tau_1 + \pi_{\textrm{VU}}\tau_2}, \\
    \text{and} \quad r^\prime &= \frac{\pi_{\textrm{VU}}(\tau_2 + \tau_3)}{\tau_1 + \pi_{\textrm{VU}}\tau_2}.
\end{align*}

These solutions were computed in~\cite{emergency15} by reduction to a fixed point problem in a
space of germs of affine functions, lexicographically ordered. This involves a subdivision of the space of resources as a polyhedral complex, and checking whether every cell of this complex leads to a feasible steady state. In this paper, we provide {\em a priori\/} conditions to avoid these steps for establishing existence.

\section{Sufficient conditions for the existence of invariant half-lines}\label{sec:main}

We seek stationary solutions of the dynamics~\eqref{eq:tds} under the form of \emph{invariant half-lines}, \ie, functions of the form 
\begin{equation*}
x(t) = u + \rho (t + t_1) \quad (u, \rho \in \R^n \, , \; t_1 \geq 0) \,,
\end{equation*}
that satisfy
\begin{equation*}
u_i + \rho_i (t + t_1) = \min_{a \in \actions_i} \biggl( r_i^a + \sum_{\tau \in \timedelays} \sum_{j \in [n]}(P_{\tau}^a)_{ij} (u_j + \rho_j (t+t_1-\tau)) \biggr) \,,
\end{equation*}
for all $t \geq 0$ and $i \in [n]$.

To the time-delay system \eqref{eq:tds} we associate a parametric family of {\em discounted problems}, which, for each \(\alpha\in [0,\infty)\), consists in finding \(v(\alpha)\in \R^n\) such that
\begin{equation}\label{eq:discounted}
  v_i(\alpha) = \min_{a\in \actions_i} 
  \big( r_i^a + \sum_{\tau \in \timedelays} \sum_{j\in [n]}{(P_\tau^a)}_{ij} \alpha^{\tau} v_j(\alpha) \big)
  \quad i\in [n] \tag{\(D_\alpha\)}\,.
\end{equation}
We use the term {\em discount factor\/} for the parameter $\alpha$, by extension of the classical situation of MDPs. Recall that, in the latter case,
we have $0\leq \alpha<1$, and $v_i(\alpha)$ yields the value of the $\alpha$-discounted MDP with initial state $i$, see~\cite{puterman}. More generally, when $\timedelays$ is a subset of $\N$ and is not reduced to $\{1\}$, $v_i(\alpha)$ is the value of a discounted semi-Markov decision process~\cite[Chap. 11]{puterman}.
    
A {\em policy\/} \(\sigma \colon [n]\to \cup_{i\in [n]} \actions_i\) is a map such that \(\sigma(i)\in \actions_i\) for all \(i\). It is a way of encoding which minimum is chosen in each component \(i\) in the piecewise linear system. To every policy \(\sigma\) we associate a vector \(r^\sigma \in \R^n\), and for each \(\tau \in \timedelays\) a matrix \({P_\tau^\sigma \in \R^{n \times n}}\) such that
\begin{equation*}
  r_i^\sigma\coloneqq {(r^{\sigma(i)})}_i,  \qquad {(P_{\tau}^{\sigma})}_{ij} \coloneqq {(P_{\tau}^{\sigma(i)})}_{ij}\,.
\end{equation*}
Specifying a policy \(\sigma\) allows us to define the {\em matrix polynomial}
\[
P^\sigma(\alpha) \coloneqq \sum_\tau P_\tau^\sigma \alpha^\tau \in \R^{n\times n}[\alpha]
\enspace ,
\]
which can be thought of either as a formal polynomial in an {\em indeterminate} $\alpha$, with matrix valued
coefficients, or if one prefers, as a matrix-valued polynomial function of the
{\em variable} $\alpha$.
We denote by \({[P^a(\alpha)]}_i\) the \(i\)-th row vector of such a matrix polynomial
where \(\sigma(i) = a\). With this notation, we can write the discounted equations~\eqref{eq:discounted} as
\begin{align*}
  v_i(\alpha) = \min_{a \in \actions_i} \biggl( r_i^a + {[P^a(\alpha)]}_i v(\alpha) \biggr) \quad  \forall i \in [n]\,.
\end{align*}

We denote by \( \Laurent \) the set of real rational functions of \( \alpha \) with a pole of order at most \(1\) at point \(\alpha=1\). Such a function has a Laurent series expansion of the form
\begin{equation*}
f(\alpha) = \frac{\lc_{-1}}{1-\alpha} + \lc_0 + (1-\alpha) \lc_1 + \dots
\end{equation*}
with \( \lc_{-1},\lc_0, \lc_1, \ldots \in \R \), which is absolutely converging in a complex punctured disk centered at point $1$. (By {\em punctured}, we mean
that the center of the disk is excluded.)
Our proof strategy is to show that for all $i\in[n]$, $v_i(\alpha)$ belongs to $\Laurent$,
and to derive the existence of a steady state from this fact. This is inspired
by the special case of MDPs, where the fact that $v_i(\cdot)$ belongs
to $\Laurent$ is a key ingredient of the proof of existence of an invariant
half-line, see~\cite{dynkin1979controlled} and~\cite{Koh80}.

\subsection{The One Policy Case}

We first consider the situation in which \(\lvert \actions_i \rvert =1\)
for all \(i\in [n]\), so that the minimum in~\eqref{eq:tds} becomes trivial and there is only one policy.
Since $\actions_i$ has only a single element, we omit the index of actions
``$a$'', writing for instance $r_i$ instead of $r_i^a$.

We start by recalling a classical fact about resolvents. Recall that an eigenvalue
is {\em semisimple} if its algebraic multiplicity coincides with its geometric multiplicity.
\begin{proposition}[{\cite[p.~40]{kato}}]\label{prop:kato}
  Let \(M\in \R^{n\times n}\).
  Suppose that either \(1\) is not an eigenvalue of \(M\) or that it is a semisimple
  eigenvalue of \(M\). Then, there exists a sequence
  of matrices \(\LC_{-1},\LC_0,\LC_1,\dots\in \R^{n\times n}\) such that
  \[
  {(I -\alpha M)}^{-1}  = \frac{\LC_{-1}}{1-\alpha} + \LC_0  + (1-\alpha)\LC_1 + \dots
  \]
  the expansion being absolutely convergent in a sufficiently small
  punctured disk of center \(1\).
\end{proposition}

\begin{theorem}\label{thm:semisimple}
  Let \(P \in  {(\R[\alpha])}^{n \times n}\) be a matrix polynomial and \(r \in \R^n\) a vector. Suppose that \(1\) either is not an eigenvalue of the matrix \(P(1)\) or that it is a semisimple
  eigenvalue of \(P(1)\). Let \(\LC_{-1},\LC_0,\dots\) be as in~\Cref{prop:kato}
  with \(M\coloneqq P(1)\).
  Let \[S= \sum_{\tau \in \timedelays} (\tau-1) P_\tau \] and suppose in addition that \(I+\LC_{-1}S\) is invertible. Then, for \(\alpha\) in a sufficiently small punctured disk of center \(1\) the system of equations
  \[v = r +P(\alpha)v\] has a unique solution \(v(\alpha)\). Moreover, the entries of \(v(\alpha)\) are rational functions belonging to \(\mathcal{L}_{-1}\).
\end{theorem}

\begin{proof}
Suppose that \(v\) satisfies the above system. Then, denoting \(P(0) = P_0\), \(P(\alpha) - P(0) = \alpha \sum_{\tau > 0} P_\tau \alpha^{\tau-1}\) and \(P(1) - P(0) =  \sum_{\tau > 0} P_\tau\) we get

\begin{align*}
  v & = r + \bigg(P(\alpha) - P(0) + \alpha (P(0) - P(1)) + \alpha P(1) + (1-\alpha)P(0)\bigg) v 
  \\
  & = r + \left(\alpha \sum_{\tau > 0}P_\tau \alpha^{\tau-1} - \alpha \sum_{\tau > 0}P_\tau +\alpha P(1) + (1-\alpha) P(0)\right) v 
  \\
  & = r + \left(\alpha\sum_{\tau > 0}P_{\tau}(\alpha^{\tau-1}-1) +\alpha P(1) + (1-\alpha)P(0)\right) v \,.
\end{align*}
Rearranging and noting that \((\alpha^{\tau-1} - 1) = (\alpha - 1)(1 + \cdots + \alpha^{\tau-2})\) we get 
  \begin{align*}
    \Big(I - \alpha P(1)  + (1-\alpha) \Big[\alpha \sum_{\tau > 0} (1 + \cdots  + \alpha^{\tau-2}) P_\tau  - P(0) \Big]\Big)v & = r  \,,
  \end{align*}
  which in turn gives 
  \begin{align*}
    \Big(I - \alpha P(1)  + (1-\alpha) \Big[\alpha \sum_{\tau > 0} (\tau - 1) P_\tau + O(1-\alpha)- P(0)\Big]\Big)v & = r  \,,
  \end{align*}
  where the term \(O(1-\alpha)\) is a matrix of rational functions (here, polynomials) that cancels at \(\alpha = 1\). Remark that \(\alpha \sum_{\tau > 0} (\tau - 1) P_\tau = \sum_{\tau > 0} (\tau - 1) P_\tau + O(1-\alpha)\) for \(\alpha\) close enough to \(1\), and \(\sum_{\tau > 0} (\tau - 1) P_\tau - P(0) = \sum_{\tau \geq 0} (\tau - 1) P_\tau = S\). Thus, the previous system can be written equivalently as 
  \begin{align}\label{eq:semisimple1}
    \Big(I - \alpha P(1)  + (1-\alpha) \Big[S + O(1-\alpha)\Big]\Big)v & = r  \,.   
  \end{align}
  Recall that \(1\) is either not an eigenvalue of \(P(1)\) or is a semisimple eigenvalue of \(P(1)\) so \(I-\alpha P(1)\) is invertible for \(\alpha\) in a punctured disk of center \(1\) with 
  \begin{align*}
    {(I -\alpha P(1))}^{-1}  = \frac{\LC_{-1}}{1-\alpha} + O(1) \enspace ,
  \end{align*}
  where \(O(1)\) is a matrix of rational functions without pole at \(\alpha = 1\). Multiplying~\eqref{eq:semisimple1} by \({(I-\alpha P(1))}^{-1}\) on both sides we get
  \begin{align}\label{eq:semisimple2}
    \Big(I + \LC_{-1}S + O(1 - \alpha)\Big)v & = {(I-\alpha P(1))}^{-1}r  \,. 
  \end{align}
  Since by assumption, \(I + \LC_{-1} S \) is invertible so is the left-hand side matrix of~\eqref{eq:semisimple2} provided that \(\alpha\) is close enough to \(1\). As a consequence for \(\alpha\) close enough to \(1\) the system~\eqref{eq:semisimple2} is equivalent to
  \begin{align*}
  v&= {\Big( I + \LC_{-1} S + O(1-\alpha)  \Big)}^{-1}{(I-\alpha P(1))}^{-1}r
  \\
  & = {\Big( I + \LC_{-1} S + O(1-\alpha)  \Big)}^{-1}(\frac{\LC_{-1}}{1-\alpha} +  O(1))r
  \\
  &=
\Big(\frac{1}{1-\alpha} {(I + \LC_{-1} S)}^{-1} \LC_{-1} + O(1)\Big)r \,.
  \end{align*}
We deduce that \(v\) is uniquely defined for \(\alpha\) in a punctured disk of center \(1\) and that it determines an element of \({(\Laurent)}^n\) which concludes the proof.
  \end{proof}

\subsection{The General Piecewise Linear Case}

Reintroducing the original case where \(\actions_i\) may contain several elements, we define the map \(T_\alpha \colon \R^n \to \R^n\) as
\begin{align}\label{eq:def-T}
  {[T_\alpha(v)]}_i \coloneqq
  \min_{a\in \actions_i} \big(r_i^a + {[P^a(\alpha)]}_i v \big) \quad \forall i\in [n] \,.
  \end{align}
so that the discounted problem can be rewritten as
\begin{equation*} 
   v = T_\alpha(v) \,.
\end{equation*}
Similarly for every policy \(\sigma\), we define the map \(T_\alpha^\sigma \colon \R^n\to \R^n\),
such that
\begin{equation*}
T^\sigma_\alpha(v) \coloneqq
r^\sigma + P^\sigma(\alpha)v \,.
\end{equation*}
We say that a policy $\sigma$ is {\em active\/} if for all $i\in[n]$, the set
\begin{align*}
  \biggl\{ z = (z_\tau)_{\tau \in \mathcal{T}} \in (\mathbb{R}^n)^{\mathcal{T}} \;\Big|\; 
  &\min_{a \in \actions_i} \Bigl( r_i^a + \sum_{\tau \in \mathcal{T}} (P^a_\tau)_i z_\tau \Bigr) \\
  &= r_i^{\sigma(i)} + \sum_{\tau \in \mathcal{T}} (P^{\sigma(i)}_\tau)_i z_\tau \biggr\}
  \end{align*}
is of non-empty interior. The introduction of active policies is motivated by the following selection property.
\begin{lemma}[Selection]\label{lem:sel}
  For all $(z_\tau)_{\tau\in\mathcal{T}}\in (\R^n)^{\mathcal{T}}$, there is an active policy $\sigma$ such that
  \begin{align}
  \min_{a \in \actions_i} \biggl( r_i^a + \sum_{\tau\in\mathcal{T}} (P^a_\tau)_i z_\tau \biggr) = r_i^{\sigma(i)} + \sum_{\tau\in\mathcal{T}} (P^{\sigma(i)}_\tau)_i z_\tau \,.
      \label{eq:sel}
  \end{align}
\end{lemma}
\begin{proof}
  To every policy $\sigma$, we associate the set $E^\sigma$ consisting of those vectors $(z_\tau)_{\tau\in\mathcal{T}}\in (\R^n)^{\mathcal{T}}$ such that
for
  all $i\in [n]$, 
  \begin{align*}
    \min_{a \in \actions_i} \biggl( r_i^a + \sum_{\tau\in\mathcal{T}} (P^a_\tau)_i z_\tau \biggr) = r_i^{\sigma(i)} + \sum_{\tau\in\mathcal{T}} (P^{\sigma(i)}_\tau)_i z_\tau \,.
    \end{align*}
  By construction, we have $(\R^n)^{\mathcal{T}} = \cup_{\sigma} E^\sigma$ where the union is taken over all policies. Moreover, since every set $E^\sigma$ is a closed
  polyhedron, and since every polyhedron of non-empty interior is the closure
  of its interior, we deduce that we still have
  $(\R^n)^{\mathcal{T}} = \cup_{\sigma} E^\sigma$ where now the union is taken only
  over the set of {\em active\/} policies. This entails that the selection property~\eqref{eq:sel} holds.
  \end{proof}

We shall make use of the following assumption.
\begin{assumption}\label{as:1}
  For each active policy \(\sigma\) we have that
  \begin{enumerate}[(i)]
    \item For all \(0\leq \alpha <1\), \(1\) is not an eigenvalue of \(P^\sigma(\alpha)\);\label{as:1-1}
    \item The matrix \(P^\sigma(0)\) has no eigenvalue in \([1, +\infty)\).\label{it-1}\label{as:1-2}
  \end{enumerate}
\end{assumption}
\begin{assumption}\label{as:2}
  For each active policy \(\sigma\),
  \begin{enumerate}[(i)]
    \item Either \(1\) is not an eigenvalue of \(P^\sigma(1)\) or it is a semisimple eigenvalue of \(P^\sigma(1)\);\label{as:2-1}
    \item The matrix \(I + \LC_{-1}^\sigma S^\sigma\) is invertible, where \(\LC_{-1}^\sigma\) and \(S^\sigma\) are defined as in~\Cref{thm:semisimple}, replacing \(P\) by \(P^\sigma(1)\).\label{as:2-2}
  \end{enumerate}
\end{assumption}
\begin{theorem}\label{thm:gendelays}
Suppose that~\Cref{as:1} holds.
Then, the nonlinear equation \(v=T_\alpha (v)\) has a solution \(v\in \R^n\) for all \(\alpha \in [0,1)\).
\end{theorem}
To show this theorem, we choose an arbitrary norm \(\| \cdot \|\) on \(\R^n\), and denote by
the same symbol, \(\|\cdot\|\), the induced matrix norm.

\begin{lemma}\label{lemma:gendelays0}
Assuming~\Cref{as:1}.~\eqref{as:1-2}, the system \(v = T_0(v)\) has a solution. 
\end{lemma}
The proof relies on topological degree theory for continuous maps, we refer
the reader to the exposition in~\cite[Chapter IV, section~2]{ruiz} for basic definitions and properties.
\begin{proof}[\proofname{} (\Cref{lemma:gendelays0})]
We actually prove that the system \(v = \beta T_0(v)\) has a solution for all \(\beta \in [0,1]\). Let us fix \(\beta \in [0,1]\). Remark that if the system \(v = \beta T_0(v)\) has a solution, then, by~\Cref{lem:sel}, there exists an active policy \(\sigma\) such that \(v = \beta T^\sigma_0(v)\), or equivalently, \((I - \beta P^\sigma(0)) v = \beta r^\sigma\). By~\Cref{as:1}.~\eqref{as:1-2}, the latter has a unique solution \(v =  \beta {(I - \beta P^\sigma(0))}^{-1} r^\sigma\). As a consequence, all zeros of \(I - \beta T_0\) for \(\beta \in [0,1]\) are bounded by \(R_0\coloneqq 1 + \max_\sigma \max_{0 \leq \beta \leq 1} \| {(I - \beta P^\sigma(0))}^{-1} \| \| r^\sigma \| < \infty\). 

Using the invariance of the topological degree under homotopy (see~\cite[Chapter IV, Prop.~2.4]{ruiz}),
this implies that the topological degree of the map \(I - \beta T_0\) w.r.t.~the ball of center \(0\) and radius \(R_0\) is invariant along the path \(\beta \in [0,1]\). Denoting the degree \(\deg\), we have for all \(\beta \in [0,1]\) and all \(R \geq R_0\),
\begin{equation}\label{eq:degree}
\deg(I-\beta T_{0}, B(0,R)) = \deg (I, B(0,R)) = 1 \,. \qedhere
\end{equation}
\end{proof}

\begin{proof}[\proofname{} (\Cref{thm:gendelays})]
By \Cref{lemma:gendelays0}, we only need to prove that the system \(v = T_\alpha(v)\) has a solution for \(\alpha \in (0,1)\). Consider \(\alpha \in (0,1)\). If \(v = T_\alpha(v)\), then, again by~\Cref{lem:sel}, there exists a policy \(\sigma\) such that \(v = T_\alpha^\sigma(v)\), or equivalently, \((I - P^\sigma(\alpha)) v = r^\sigma\). By~\Cref{as:1}.~\eqref{as:1-1}, the latter has a unique solution \(v = {(I - P^\sigma(\alpha))}^{-1} r^\sigma\). 

Now, let us fix \(\alpha_0 \in (0,1)\), and define 
\begin{equation*}  
R_{\alpha_0}\coloneqq
1 + \max_{\sigma} \max_{\alpha \in [0,\alpha_0]} \lVert {(I- P^{\sigma}(\alpha))}^{-1} \rVert \lVert r^\sigma \rVert < \infty \,,
\end{equation*}
Then, we have for all \(\alpha \in [0,\alpha_0]\) and \(R \geq \max(R_0, R_{\alpha_0})\), 
\begin{equation*}
\deg(I - T_\alpha, B(0, R)) = \deg(I- T_{0}, B(0, R)) = 1 \,,
\end{equation*}
where the last equality follows~\eqref{eq:degree}. Since the degree of the map \(I-T_\alpha\) with respect to \(B(0,R)\) is nonzero, is has a zero in the open ball \(B(0,R)\) (Coro.~2.5: (2), {\em ibid.\/}).
\end{proof}

We define a {\em germ\/} of a function at point \(1^-\) to be an equivalence class
under the relation which identifies two real functions which coincides
on some interval \((\alpha_0,1)\). 
We equip \(\Laurent\) with a structure of totally ordered field by defining that \(f \in \Laurent\) is positive if there exists \(\alpha_0\in (0,1)\) such that for all \(\alpha \in (\alpha_0, 1)\), \(f(\alpha) > 0\). This field is totally ordered because the zeros of a Laurent series cannot have a finite accumulation point.
In other words, the space of germs of Laurent series at point \(1^-\) is totally
ordered. Note that with this order, the operators \(T_\alpha\) and \(T^\sigma_\alpha\) can be defined as maps over \((\Laurent)^n\), and that~\Cref{eq:tds} and~\Cref{eq:discounted} make sense defined over \((\Laurent)^n\).

\begin{theorem}\label{thm:laurent}
  Suppose that both~\Cref{as:1} and~\Cref{as:2} hold.
  Then, the nonlinear equation \(v = T_\alpha(v)\) has a solution over \({(\Laurent)}^n\).
\end{theorem}

\begin{proof}
By~\Cref{thm:semisimple}, every policy \(\sigma\) defines a (unique) vector of rational functions \(v^\sigma (\alpha) \in (\Laurent)^n\) so that
\[
v^\sigma(\alpha) = \frac{\lc_{-1}^\sigma}{1-\alpha} + \lc_0^\sigma + O(1-\alpha)
\]
where \(\lc_{-1}, \lc_0 \in \R^n\), and \(v^\sigma = T^\sigma_\alpha(v^\sigma)\).

Let \(\alpha_0 \in (0,1)\) such that:
\begin{enumerate}[(a)] 
\item\label{item:assump1} for all \(\sigma\), \Cref{thm:semisimple} applies to the system \(v = T^\sigma_\alpha(v)\) over \(\R^n\) for all \(\alpha \in (\alpha_0, 1)\);
\item\label{item:assump2} for all \(\sigma\) and \(i \in [n]\), the rational function \(v^\sigma_i\) has no pole in \((\alpha_0, 1)\);
\item\label{item:assump3} for all \(\sigma, \sigma'\), the sign of the rational function \([v^\sigma - T_\alpha^{\sigma'}(v^\sigma)]_i\) is constant over \((\alpha_0, 1)\).  
\end{enumerate}

Let \(\alpha_1 \in (\alpha_0, 1)\). By \Cref{thm:gendelays}, the equation \(v=T_{\alpha_1}(v)\) admits a solution \(v_{\alpha_1}\in \R^n\). Let \(\sigma^\star\) be the active policy such that \(v_{\alpha_1} = T^{\sigma^\star}_{\alpha_1}(v_{\alpha_1})\) by~\Cref{lem:sel}.

We claim that \(v^{\sigma^\star}\) is a solution to \(v = T_\alpha(v)\) over \((\Laurent)^n\). First, by assumption~\eqref{item:assump2}, \(v^{\sigma^\star}\) is well-defined on \(\alpha = \alpha_1\). Thus, we have \(v^{\sigma^\star}(\alpha_1) = v_{\alpha_1}\) by \Cref{thm:semisimple} and assumption~\eqref{item:assump1}. Now, let \(\sigma\) be another policy. By assumption~\eqref{item:assump3}, the sign of \([v^{\sigma^\star} - T^\sigma_\alpha(v^{\sigma^\star})]_i(\alpha)\) is constant for all \(\alpha \in (\alpha_0, 1)\). Since it is nonpositive at \(\alpha = \alpha_1\), we deduce that \([v^{\sigma^\star} - T^\sigma_\alpha(v^{\sigma^\star})]_i\) is a nonpositive element of \(\Laurent\). In other words, we have \(v^{\sigma^\star} \leq T^\sigma_\alpha(v^{\sigma^\star})\) for all policies \(\sigma\), with equality when \(\sigma = \sigma^\star\).
\end{proof}

Next we present a lemma describing the correspondence between the existence of an invariant half-line of~\Cref{eq:tds} and the fulfillment of a lexicographic system.
\begin{lemma}\label{lemma-lexicographic}
  The system~\Cref{eq:tds} admits an invariant half-line \(x(t) = u + (t+ t_1) \rho\) for some \(t_1 \geq 0\) if and only if the following lexicographic system holds
  \begin{align}
    \rho_i & = \min_{a \in \actions_i} {[P^a(1)]}_i \rho \label{eq:lexeta} \,,\\
    u_i & = \min_{a \in \actions_i^\star} \left(r_i^a - \rho_i + {[P^a(1)]}_i u - {[S^a]}_i \rho \right) \label{eq:lexu} \,,
  \end{align} 
  where \(\actions_i^\star = {\arg\min}_{a \in \actions_i} {[P^a(1)]}_i \rho\) \,.
\end{lemma}
This is proved in Appendix, section~\ref{append-lemma-lexicographic}.

\begin{theorem}\label{thm:halfline}
  Suppose that the nonlinear equation \(v = T_\alpha(v)\) has
  a solution \(v\in (\Laurent)^n\), so that
  \begin{align}
    v  & = \frac{\rho}{1-\alpha} + u + O(1-\alpha) \label{eq:asymp-l}
  \end{align}
  for some \(u,\rho\in \R^n\).
  Then, the time-delay dynamical system~\eqref{eq:tds} admits an invariant half-line \(t \mapsto u + t_1 \rho + \rho t\),
  for some \(t_1 \geq 0\).
\end{theorem}
\begin{proof}
  Substituting the solution~\eqref{eq:asymp-l} into our equation \(v = T_\alpha(v)\) then for each \(i \in [n]\) we get
  \begin{equation}
    \begin{split}        
      & \frac{\rho_i}{1-\alpha} + u_i + O(1-\alpha) \\
      &= \min_{a \in \actions_i} \biggl( r^a_i + {[P^a(\alpha)]}_i \Bigl( \frac{\rho}{1-\alpha}  + u + O(1-\alpha) \Bigr) \biggr) \,.
    \end{split} \label{eq:halfline-disc}
  \end{equation}
  We have the following Taylor expansion of \({[P^a(\alpha)]}_i\) at \(\alpha = 1\)
  \begin{align*}
    & {[P^a(\alpha)]}_i \\
    & = {[P^a(1)]}_i - (1-\alpha) \sum_\tau \tau {[P_\tau^a]}_i + {O(1-\alpha)}^2 \\
    & = {[P^a(1)]}_i  - (1-\alpha) \biggl( \sum_\tau (\tau-1) {[P_\tau^a]}_i + \sum_\tau {[P_\tau^a]}_i \biggr) + {O(1-\alpha)}^2 \\
    & = {[P^a(1)]}_i  - (1-\alpha) \bigl( {[S^a]}_i + {[P^a(1)]}_i \bigr) + {O(1-\alpha)}^2 \,.
  \end{align*}
  Substituting this expansion into~\eqref{eq:halfline-disc} we get

  \begin{align}
    &\frac{\rho_i}{1-\alpha} + u_i + O(1-\alpha) \notag \\
    &= \min_{a \in \actions_i} \begin{aligned}[t]
        &\biggl( r^a_i + \Bigl({[P^a(1)]}_i - (1-\alpha)\bigl({[S^a]}_i + {[P^a(1)]}_i\bigr) \\
        &\quad + {O(1-\alpha)}^2 \Bigr) \Bigl(\frac{\rho}{1-\alpha} + u + O(1-\alpha)\Bigr) \biggr)
    \end{aligned} \notag \\
    &= \min_{a \in \actions_i} \begin{aligned}[t]
        &\biggl( \frac{{[P^a(1)]}_i \rho}{1-\alpha} + r^a_i - {[P^a(1)]}_i \rho \\
        &\quad + {[P^a(1)]}_i u - {[S^a]}_i \rho + O(1-\alpha) \biggr) \,.
    \end{aligned} \label{eq:taylor}
  \end{align}
  In an analogous manner to the proof of~\Cref{lemma:gendelays0} we will prove the equivalence of~\eqref{eq:taylor} to the lexicographic system reproduced here
  \begin{align}
    \rho_i  & = \min_{a\in \actions_i} {[P^a(1)]}_{i} \rho\label{eq:lexeta-halfline} \,,\\
    u_i &= \min_{a\in \actions_i^\star} \Big( r_i^a + {[P^a(1)]}_{i} u - {[S^a]}_i \rho + {[P^a(1)]}_{i} \rho \Big) \,, \label{eq:lexu-halfline}
  \end{align}
  where, as before, \(\actions_i^\star\) is defined as \(\arg\min_{a \in \actions_i}{[P^a(1)]}_{i} \rho\).

  First, multiplying by \(1-\alpha\) on both sides of~\eqref{eq:taylor} gives
  \begin{align*}
    \rho_i + O(1-\alpha) = \min_{a \in \actions_i} \left( {[P^a(1)]}_i \rho + O(1-\alpha) \right) \,,
  \end{align*}
  where taking the limit \(\alpha \to 1\) shows that~\eqref{eq:lexeta-halfline} is satisfied.

  Next, we rewrite~\eqref{eq:taylor} as follows,
  \begin{align*}
    u_i = \min_{a \in \actions_i} &\biggl( \frac{{[P^a(1)]}_i \rho - \rho_i}{1-\alpha} + r^a_i - {[P^a(1)]}_i \rho \\ 
    &\quad + {[P^a(1)]}_i u - {[S^a]}_i \rho + O(1-\alpha) \biggr) \,.
  \end{align*}
  By equation~\eqref{eq:lexeta-halfline}, the factor \({[P^a(1)]}_i \rho - \rho_i\) is strictly positive for every \(a \not\in \actions_i^\star\) so the term \(\frac{{[P^a(1)]}_i \rho - \rho_i}{1-\alpha}\) goes to \(+\infty\) as \(\alpha\) converges to \(1\) from below. Thus the equation must be restricted to the case where \(a \in \actions_i^\star\) and the factor is zero. This implies that
  \begin{align*}
    u_i = \min_{a\in \actions_i^\star} \bigl( r_i^a + {[P^a(1)]}_{i} u - {[S^a]}_i \rho + {[P^a(1)]}_{i} + O(1-\alpha) \rho \bigr) \,,
  \end{align*}
  where taking the limit \(\alpha \to 1^-\) gives that~\eqref{eq:lexu-halfline} is satisfied.

  With the satisfaction of the lexicographic system~\eqref{eq:lexeta-halfline} and~\eqref{eq:lexu-halfline} we can apply~\Cref{lemma:gendelays0} which completes the proof.\qedhere

\end{proof}
As an immediate consequence of~\Cref{thm:laurent} and~\Cref{thm:halfline}, we now arrive
at our main result:
\begin{theorem}\label{cor:main}
  Under Assumptions~\ref{as:1} and~\ref{as:2}, the piecewise linear dynamical system with priorities~\eqref{eq:tds} admits a stationary solution.
\end{theorem}
\section{Extending Kohlberg's invariant half-line theorem}\label{sec:kohl}
Our method leads to an extension of Kohlberg's result, showing that
a piecewise linear nonexpansive map admits an invariant half-line.
To present this extension, we now consider the more general class of dynamics given by
\begin{align}\label{eq:tds2}
  x_i(t) = \max_{b\in\mathcal{B}_i} \min_{a\in \actions_{i,b}} \biggl( r_i^{ab} + \sum_{\tau\in \timedelays} \sum_{j\in [n]} {(P_\tau^{ab})}_{ij} x_j(t-\tau ) \biggr)\,, \tag{\(D^*\)}
\end{align}
for all \(i \in [n]\) and \(t \geq 0\). The relevance of this class of systems is motivated by a result of Ovchinnikov~\cite{ovchinnikov} who shows that every piecewise linear map can be expressed as a max/min combination of linear maps, like in the right-hand side of~\eqref{eq:tds2}. Invariant half-lines are defined similarly to the previous case, \ie, as functions $z(t) = u + \rho (t + t_1)$ (with $u, \rho \in \R^n$, $t_1 \geq 0$) that are left invariant by the dynamics~\eqref{eq:tds2} for all $t \geq 0$.

We proceed in a manner analogously to the so far dealt with minimization case. First, we define the notion of policies $\sigma \colon \cup_{i\in [n]} (\{i\}\times \mathcal{B}_i) \to \cup_{i\in[n]} \mathcal{A}_i$
of Min and $\pi \colon [n]\to \cup_{i\in [n]}\mathcal{B}_i$ of Max and a pair of policies defined by the pair \((\sigma, \pi)\). Note that the minimization is done with regards to what action has been taken for the maximization but not vice-versa, this is what is adapted to in our extended definitions.
Then, we can define an {\em active joint policy\/} as a joint policy \((\sigma, \pi)\) such that the set of vectors $z=(z_\tau)_{\tau \in \mathcal{T}}\in (\R^n)^{\mathcal{T}}$ fulfilling
\begin{align*}
  &\max_{b\in\mathcal{B}_i} \min_{a\in \actions_{i,b}} \biggl( r_i^{ab} + \sum_{\tau\in \timedelays}  (P_\tau^{ab})_{i} z_\tau  \biggr)\\
  &= r_i^{\sigma(i,\pi(i)),\pi(i)} + \sum_{\tau\in \timedelays}  \Bigl(P_\tau^{\sigma(i,\pi(i)),\pi(i)}\Bigr)_{i} z_\tau
  \end{align*}
has a non-empty interior for all \(i\in[n]\).

Further, we can formulate a new selection lemma. We skip
the proof as it is an immediate variation on the original proof.
\begin{lemma}[Nested Selection]\label{lem:joint-sel}
  For all $(z_\tau)_{\tau\in\mathcal{T}}\in (\R^n)^{\mathcal{T}}$, there is an active joint policy \((\sigma, \pi)\) such that
  \begin{align}
    &\max_{b\in\mathcal{B}_i} \min_{a\in \actions_{i,b}} \Bigg( r_i^{ab} + \sum_{\tau\in \timedelays}  (P_\tau^{ab})_{i} z_\tau  \Bigg) \notag \\
    &= r_i^{\sigma(i,\pi(i)),\pi(i)} + \sum_{\tau\in \timedelays} {\bigl(P_\tau^{\sigma(i,\pi(i)),\pi(i)}\bigr)}_{i} z_\tau \,. \label{eq:joint-sel}
  \end{align}
\end{lemma}
We also extend the assumptions~\Cref{as:1} and~\Cref{as:2} to
\begin{assumptionprime}\label{as:1prime}
  For each active joint policy \((\sigma, \pi)\) we have that
  \begin{enumerate}[(i)]
    \item For all \(0\leq \alpha <1\), \(1\) is not an eigenvalue of \(P^{\sigma \pi}(\alpha)\),\label{as:1-1prime}
    \item The matrix \(P^{\sigma \pi}(0)\) has no eigenvalue in \([1, +\infty)\);\label{as:1-2prime}
  \end{enumerate}
\end{assumptionprime}
\begin{assumptionprime}\label{as:2prime}
  For each active joint policy \((\sigma, \pi)\),
  \begin{enumerate}[(i)]
    \item either \(1\) is not an eigenvalue of \(P^{\sigma \pi}(1)\) or it is a semisimple eigenvalue of \(P^{\sigma \pi}(1)\); \label{as:2-1prime}
    \item the matrix \(I + \LC_{-1}^{\sigma \pi} S^{\sigma \pi}\) is invertible, where \(\LC_{-1}^{\sigma \pi}\) and \(S^{\sigma \pi}\) are defined as in~\Cref{thm:semisimple}, replacing \(P\) by \(P^{\sigma \pi}\). \label{as:2-2prime}
  \end{enumerate}
\end{assumptionprime}
The following theorem extends~\Cref{cor:main} to the ``max-min'' case.
\begin{theorem}\label{thm:maxmin}
  Suppose that Assumptions~\ref{as:1prime} and~\ref{as:2prime} hold. Then, the dynamical system~\eqref{eq:tds2} admits a stationary solution.
\end{theorem}
\begin{proof}
  The proof is identical to the minimization case except that we use the joint selection property~\Cref{lem:joint-sel} in the places where~\Cref{lem:sel} was invoked. 
  \end{proof}
We next recover Kohlberg's theorem on invariant half-lines
as a special case of~\Cref{thm:maxmin}.
\begin{corollary}[{\cite{Koh80}}]\label{cor:kohl}
  Suppose that $T$ is a piecewise affine self-map of $\R^n$ that
  is nonexpansive in a norm $\|\cdot\|$. Then, $T$ admits an invariant
  half-line.
\end{corollary}  
\begin{proof}
  By~\cite{ovchinnikov}, every coordinate of a piecewise linear function $T$
  can be written as
  \begin{equation*}
  T_i(x) = \max_{b\in \mathcal{B}_i} \min_{a\in \actions_{i,b}} (r_i^{ab} + \sum_{j\in[n]}(P^{ab})_{ij} x_j) \,,
  \end{equation*}
  where for all \(i\in [n]\), \(\mathcal{B}_i\) is a finite set,
  and for all \(b\in \mathcal{B}_i\), \(\actions_{i,b}\) is also a finite
  set, and \(r_i^{ab}\) and and \(P^{ab}_{ij}\) are real numbers.
  It also follows from~\cite{ovchinnikov}
  that the representation can be chosen
  so that for all $i\in[n]$, $b\in\mathcal{B}_i$, $a\in\actions_{i,b}$,
  we have $T_i(x) = r_i^{ab} + \sum_{j\in[n]}(P^{ab})_{ij} x_j$ for all
  $x$ belonging to a polyhedron of non-empty interior.

  In essence, for each joint policy \((\sigma, \pi)\) we are left with a linear extension of a linear function
  \begin{equation*}
    T_i^{\sigma(\pi(i)), \pi(i)} = r_i^{\sigma(\pi(i)), \pi(i)} + \sum_{j\in[n]}(P^{\sigma(\pi(i)), \pi(i)})_{ij} x_j
  \end{equation*}
  which, by virtue of extending an affine operator from a full-dimensional subset to the entire space, inherits the property of nonexpansiveness, implying that \(P^{\sigma(\pi(i)), \pi(i)}\) is nonexpansive for every joint policy \((\sigma, \pi)\).
  
  We restrict our view to unit delays, implying that \(0 \not\in \timedelays\). By the simple observation then that \(\alpha P^{\sigma(\pi(i)), \pi(i)}\) is an \(\alpha\)-contraction for all \(\alpha \in [0, 1)\) we get assumption~\Cref{as:1prime}.

  Next we present a lemma implying that nonexpansiveness gives the remaining assumptions.

  \begin{lemma}\label{lemma:simple}
    The spectrum of a nonexpansive operator \(P:\R^n \to \R^n\) either does not contain \(1\) or \(1\) is a semisimple eigenvalue of \(P\).
  \end{lemma}
  This is proved in Appendix~\ref{append:lemma:simple}.
  
  Applying \Cref{lemma:simple}, we instantly get that Assumption~B'.(\ref{as:2-1prime})
  is satisfied. Moreover, since $\timedelays=\{1\}$, Assumption~B'.(\ref{as:2-2prime})
  is trivially satisfied. So, \Cref{cor:kohl} follows
  from~\Cref{cor:main}.
\end{proof}

\section{Application to timed Petri nets with priorities}\label{sec:automaticPetri}
As described in~\S\ref{sec:motivation}, we are interested in the counter functions, \({(z_q)}_{q \in \mathcal{Q}}\) associated to the transitions \(q\) of timed Petri nets with priorities in order to describe their dynamical properties.
These equations can in general be reduced to the cases shown in~\Cref{tab-counter}, as detailed in~\cite[Chapter~1]{boyet2022}. There, $q\inc$ stands for the set of upstream places of a transition $q$, $p\inc$ for the set of upstream transitions of a place $p$. Moreover, $\mathcal{Q}_{\text{psel}}$ and $\mathcal{Q}_{\text{prio}}$ correspond to the (disjoint) sets of transitions respectively involved in preselection and priority rules.

\ifbool{tabularray}{%
\begin{table}[h]
  \begin{tblr}{
    colspec={ll}, 
    rowspec={|Q[headergray]|Q[tablegray]|Q[tablegray]|Q[tablegray]},
    hlines
    }
      Type & Counter equation \\
      \( \displaystyle q \notin (\mathcal{Q}_{\text{psel}} \cup \mathcal{Q}_{\text{prio}})  \) & \( \displaystyle z_{q} (t) = \min_{p \in q^{\text{in}}} \left( r_q^p + \sum_{q^\prime \in p^{\text{in}}} z_{q^\prime} (t - \tau_p) \right) \) \\
      \( \displaystyle q \in \mathcal{Q}_{\text{psel}} \) & \( \displaystyle z_{q} (t) = r_q^p + \sum_{q^\prime \in p^{\text{in}}} \pi_q^p z_{q^\prime} (t - \tau_p) \) \\
      {\( \displaystyle q \in \mathcal{Q}_{\text{prio}} \) & \( \displaystyle z_{q} (t) = \min_{p \in q^{\text{in}}} \Bigg( r_q^p + \sum_{q^\prime \in p^{\text{in}}} z_{q^\prime} (t - \tau_p) \)\\ 
                                                             \( \displaystyle - \sum_{q^\prime \prec_{p} q} z_{q^\prime} (t) - \sum_{q^\prime \succ_{p} q} z_{q^\prime} (t - \epsilon) \Bigg) \)}
  \end{tblr}
  \caption{General counter equations for the types of transitions in a Petri net with priorities \cite{boyet2022}.}
  \label{tab-counter}
\end{table}}{}

Collecting terms with the same time-delays, identifying \(p \in q^{\text{in}}\) with \(a \in \actions_i\) by replacing \(a\) by \(p\), \(i\) by \(q\), and \(\actions_i\) by \(q^{\text{in}}\), identifying the constants before the counter functions, \((\{1, \pi_q^p, -1\})\), with the coefficients \({(P_\tau^a)}_{ij}\), we get a time-delay system of the same form as~\eqref{eq:tds}.

\begin{proposition}
  A timed Petri net with priority is such that, for each policy, the discounted variant of its transition matrix, \(P^\sigma(\alpha)\) has no eigenvalue larger than \(1\) for \(\alpha = 0\). Thus its dynamics fulfills~\Cref{as:1}.\eqref{as:1-2}.
\end{proposition}

\begin{proof}
  Let \(\sigma\) be any policy. Since, by assumption, the priority orders are consistent, there exists a total order on all the transitions, we call it ''\(\prec\)''. 

  We derive the discounted equations corresponding to the Petri net counter equations in~\Cref{tab-counter}, as shown in~\Cref{tab-discounter}.
\ifbool{tabularray}{%
  \begin{table}[h]
    \begin{tblr}{colspec={ll}, rowspec={|Q[headergray]|Q[tablegray]|Q[tablegray]|Q[tablegray]}}
        Type & Discounted counter equation \\
        \( \displaystyle q \notin (\mathcal{Q}_{\text{psel}} \cup \mathcal{Q}_{\text{prio}})\)              & \( \displaystyle v_{q} (\alpha) = \min_{p \in q^{in}} \biggl( r_q^p + \sum_{\tau_p \in \mathcal{T}}\sum_{q^\prime \in p^{\text{in}}} \alpha^{\tau_p} v_{q^\prime}(\alpha) \biggr) \) \\
        \( \displaystyle q \in \mathcal{Q}_{\text{psel}}\)              & \( \displaystyle v_{q} (\alpha) =  r_q^p + \sum_{\tau_p \in \mathcal{T}}\sum_{q^\prime \in p^{\text{in}}} \pi_q^p \alpha^{\tau_p}v_{q^\prime} (\alpha) \) \\
        \( \displaystyle q \in \mathcal{Q}_{\text{prio}}\)              & { \( \displaystyle v_{q} (\alpha) = \min_{p \in q^{in}} \biggl( r_q^p + \sum_{\tau_p \in \mathcal{T}}\sum_{q^\prime \in p^{\text{in}}} \alpha^{\tau_p} v_{q^\prime}(\alpha) \) \\ 
                                                                            \( \displaystyle+ \sum_{q^\prime \prec_{p} q} (-1) \alpha^{0} v_{q^\prime} (\alpha) + \sum_{q^\prime \succ_{p} q} (-1) \alpha^\epsilon v_{q^\prime} (\alpha)  \biggr) \) }
    \end{tblr}
    \caption{Discounted problems corresponding to the general counter equations for a Petri net with priorities}
    \label{tab-discounter}
  \end{table}}{}
  We look directly at the case when \(\alpha = 0\). The only terms with a nonzero factor then are the resources, \(r\), and the terms with zero exponent \(\tau = 0\). By comparison with~\Cref{eq:discounted} we see that these solutions have the coefficient \({[P^\sigma(0)]}_{qq^\prime} = {(P^\sigma_0)}_{qq^\prime} = -1\) only if \(q' \prec_p q\), otherwise the coefficient is zero. To be precise, the coefficient may still be zero depending on the policy \(\sigma\). By ordering the matrix according to our total order \(\prec\) we get that \(P^\sigma(0)\) is a lower triangular matrix (or upper triangular matrix depending on if we order the transitions in increasing or decreasing priority) with zero diagonal and thus nilpotent. Consequently, the matrix has no eigenvalue with magnitude \(1\) or greater which completes the proof.
\end{proof}

We say that a Petri net admits a {\em positive stoichioemetric invariant}, 
if there is a positive vector \(e\) such that for
all \(i\in[n]\) and actions \(a\in A_i\), \(\sum_\tau P_{i,\tau}^a  e = e_i\).
\begin{proposition}\label{prop-nonneg}
  Suppose that a Petri net with priority admits a {\em positive stoichioemetric invariant}, and that for every policy \(\sigma\), the matrix $P^\sigma(1)$
  is nonnegative. Then, \(1\) is a semisimple eigenvalue
  of the matrix \(P^\sigma(1)\), so that~\Cref{as:2}.\eqref{as:2-1}
holds for its dynamics. 
  \end{proposition}
\begin{proof}
Select any policy \(\sigma\) and let \(P\coloneqq P^\sigma(1)\).
It follows from \(Pe=e\) and \(P\) nonnegative that \(P\)
is equivalent to a stochastic matrix by a diagonal scaling
more precisely \((e_i^{-1} P_{ij} e_j)\) is stochastic,
and so \(1\) is a semisimple eigenvalue of \(P\),
by \cite[p.~42]{bermanandplemmons}.
\end{proof}
\begin{remark}
  In all our examples, in the reduced dynamics,
  all the matrices $P^\sigma(1)$ are trivially  
  nonnegative
  (the negative terms arising from priorities cancel in the sum over \(\tau\)),
  so that \Cref{prop-nonneg} can be applied.
\end{remark}
In the sequel, the term {\em generic} means ``everywhere except in a finite union of algebraic hypersurfaces''. 
\begin{proposition}
  Suppose that a Petri net with priority admits a positive stoichioemetric invariant and that the values of the parameters \(\tau\in \timedelays\) are generic,
  and suppose in addition that for all  policies $\sigma$, the matrix
  \(P^\sigma(1)\) is nonnegative.
Then, for every policy \(\sigma\), the matrix \(I + \LC_{-1}^\sigma S^\sigma\) is invertible, where \(\LC_{-1}^\sigma\) and \(S^\sigma\) are defined as in~\Cref{thm:semisimple}, replacing \(P\) by \(P^\sigma(1)\). So the dynamics generically satisfies~\Cref{as:2}.\eqref{as:2-2}.
\end{proposition}
\begin{proof}
For all policies \(\sigma\), consider the expansion
\[
(1-\alpha P^\sigma(1))^{-1} = \frac{C^\sigma_{-1}}{1-\alpha} +O(1) \,.
\]
We will show that the matrix
\[
I + C^\sigma_{-1} S^\sigma 
\]
is invertible for generic values of the collection of parameters $\tau$ with $\tau\in\timedelays$.
After making a diagonal scaling, we assume,
without loss of generality, that \(e\) is the unit vector (all entries one).
Then, the matrix \(P^\sigma(1)\) is a stochastic matrix, whose spectral
projector can be computed as follows using standard results on Markov chains,
see~\cite{bermanandplemmons} and especially~\cite[9.4]{rothblum}.
Let \(F_1,\dots,F_r\) denote the final classes of the digraph of \(P^\sigma(1)\),
then, \(C^\sigma_{-1}\) is a matrix of rank \(r\), with
\begin{align}
  C^\sigma_{-1} = \sum_{i\in [r]}  \pi_i  m_i
  \label{eq:specproj}
\end{align}
where \(m_i\) is unique invariant measure of \(P^\sigma(1)\) supported
by \(F_i\) (\(m_i\) row vector of dimension \(n\) but all entries outside \(F_i\) are zero),
and \(\pi_i\in \R^n\) is a potential vector---when \(e\) is the unit vector,
the entry \((\pi_i)_j\) is the probability of absorption 
by the final class \(F_i\) starting from state \(j\).

Let us begin by considering the case when there is a
single final class. Then, we next show that the invertibility
property even holds {\em universally} (without any genericity condition).
we have
\[
C^\sigma_{-1} = e m_1  \enspace .
\]
The matrix \(I + C^\sigma_{-1} S^\sigma \), which is  rank one perturbation
of the identity, is invertible
iff \(-1\) is not an eigenvalue of
\[      C^\sigma_{-1} S^\sigma
= e m_1 \sum_\tau P_\tau^\sigma(\tau-1)
\]
This matrix is of the form \(ev\) where \(v=m_1 \sum_\tau P_\tau^\sigma(\tau-1)\) is a row vector, so it is rank one, and, using cyclic conjugacy,
we see that its only possibly nonzero eigenvalue
is given  by
\[
ve= m_1 \sum_\tau P_\tau^\sigma(\tau-1) e
\]
so the matrix \(I + C^\sigma_{-1} S^\sigma \) is invertible iff
\begin{align}
-1 \neq \sum_{j\in F_1, k\in [n]} \sum_{\tau}  (m_1)_j (P_\tau^\sigma)_{jk}(\tau -1)\enspace .\label{eq:neq}
\end{align}
In the RHS, all the terms \((m_1)_j (P_\tau^\sigma)_{jk}(\tau -1)\) with \(\tau\geq 1\) are nonnegative.
For \(\tau=0\), supposing that all the places have positive integer holding times, then
\({(P_0^\sigma)}_{jk}\) can only take the values \(0\) and \(-1\), which entails that~\eqref{eq:neq}
is satisfied.

Suppose now that there are several final classes. We have
\[
C^\sigma_{-1}S^\sigma = \sum_{i\in [r]}\sum_{\tau\in \timedelays} \pi^i m_i P^\sigma_\tau (\tau-1) \enspace .
\]
Setting $\timedelays=\{\tau_1,\dots,\tau_k\}$, where $\tau_1,\dots,\tau_k$ are now thought of as formal indeterminates, we see that $\det(I + C^\sigma_{-1}S^\sigma)$ is a nonzero polynomial
in the indeterminates $\tau_1,\dots,\tau_k$. So its vanishing locus
is a (possibly empty) algebraic hypersurface, which shows that
\Cref{as:2}.\eqref{as:2-2} holds
generically.\qedhere
\end{proof}

\section{Examples of dynamical systems with priorities as Petri nets}\label{sec:catalog}
Now we apply these results to examples from previous applications. These examples generally come in the form of timed Petri nets with priority and preselection rules as described in~\S\ref{sec:motivation}. Further, in~\S\ref{sec:automaticPetri} we prove that such systems fulfill all conditions in Assumptions~\ref{as:1} and~\ref{as:2} but the one in~\Cref{as:1}.~\eqref{as:1-1}. This last assumption is instead manually confirmed for these specific examples, as we explain below. (Because of space constrains, we show the calculations explicitly only for the PFAU example.)

\subsection{The 17-18-112 PFAU Call Center Revisited}
The complete and reduced dynamics for the bi-level call center of the PFAU presented in~\cite{emergency15} are reproduced in~\eqref{eq:bilevel_simple}. These equations indeed coincide with those of a piecewise linear matrix polynomial system of the form~\eqref{eq:tds}. We represent the policies as tuples \((i, j, k)\) where each index can take one of two values, $1$ or $2$, depending on which of the first or second term in the corresponding minimum is selected. We then derive the discounted equations and through them the transition matrices, \(P^\sigma(\alpha)\). Now we can manually check Assumption~\ref{as:1}.~\eqref{as:1-1}. One way to analytically examine this is to look at the determinant \(\det (P^\sigma(\alpha)-I)\) for \(\alpha \in [0, 1)\) and each policy \(\sigma\). To illustrate, the policy \((2,2,1)\) refers to the system
\begin{equation*}
  \begin{array}{lcl}
  z_1(t) & = & N_{A} + (1-\pi_{\textrm{VU}})z_1(t-\tau_1) + z_3^\prime(t-\tau_2), \\
  z_3(t) & = & N_{P} + z_3(t-\tau_2) + z_3^\prime(t-\tau_2-\tau_3) - z_3^\prime(t), \\[1ex]
  z_3^\prime(t) & = & \pi_{\textrm{VU}}z_1(t-\tau_1). \\
  \end{array}
\end{equation*}
With the discounted transition matrix
\begin{equation*}
  P^{(2,2,2)}(\alpha) = 
  \begin{pmatrix}
    \alpha^{\tau_1} & 0 & \alpha^{\tau_2} \\
    0 & \alpha^{\tau_2} & \alpha^{\tau_2 + \tau_3} - 1 \\
    0 & \alpha^{\tau_2} - \alpha^{\epsilon} & \alpha^{\tau_2 + \tau_3}
  \end{pmatrix}\,,
\end{equation*}
and the factorized determinant
\begin{equation*}
  \det{(P^{(2,2,2)}(\alpha)-I)} = (\alpha^{\tau_1}-1)(\alpha^\epsilon-1)(\alpha^{\tau_2 + \tau_3}-1)\,,
\end{equation*}
which is nonzero by the strict positivity of the time-delays.

We get similar results for the other policies where. We only list the factorized determinants:
\begin{align*}
  & \det{(P^{(1,1,1)}(\alpha)-I)} = -1 \\
  & \det{(P^{(1,1,2)}(\alpha)-I)} = \alpha^{\tau_2+\tau_3}-1 \\
  & \det{(P^{(1,2,1)}(\alpha)-I)} = \alpha^{\tau_2}-1 \\
  & \det{(P^{(1,2,2)}(\alpha)-I)} = -(\alpha^\epsilon-1)(\alpha^{\tau_2 + \tau_3}-1) \\
  & \det{(P^{(2,1,1)}(\alpha)-I)} = \pi_{VU}\alpha^{\tau_1}\alpha^{\tau_2}(\alpha^{\tau_1}-1) \\
  & \det{(P^{(2,1,2)}(\alpha)-I)} = -(\alpha^{\tau_1}-1)(\alpha^{\tau_2 + \tau_3}-1) \\
  & \det{(P^{(2,2,1)}(\alpha)-I)} = \pi_{VU}\alpha^{\tau_1}\alpha^{\tau_2}(\alpha^{\tau_1}-1)(\alpha^{\tau_2}-1)^2\,.
\end{align*}
We see that they are all factors of nonzero elements implying that Assumption~A.(\ref{as:1-1}) is fulfilled. Note here that we have checked every policy rather than just the active ones.

\subsection{Traffic on Circular Roads with the Right-Hand Rule~\cite{farhi2019dynamicprogrammingsystemsmodeling}}
Another example is traffic on two circular one-way roads with an intersection, as described in~\cite{farhi2019dynamicprogrammingsystemsmodeling}. When two cars attempt to enter the intersection simultaneously, the car approaching from the ``right'' has priority, following the right-hand rule. Here, we create a Petri net with priorities to model this scenario. This Petri net is based on the one in~\cite{farhi2019dynamicprogrammingsystemsmodeling} but is simplified by not tracking the cars' approximate positions on the roads (\ie, we only concern ourselves with the congestion pertaining to the crossing itself, not how it may affect the traffic on the roads.). Instead, we assume a fixed time to complete a circuit around each road, tracking only the number of cars on each road, and using a priority rule to represent the right-hand rule, rather than the negative weights used in the cited work.

Specifically, imagine a north-south road intersecting an east-west road, where each road forms a circular loop. For example, a driver heading south at the intersection will circle back to approach from the north. The availability of the crossing is represented by the presence of tokens in a place, with the initial availability being determined by the resource \(A \in \{0, 1\}\). In case of conflict, this token prioritizes cars from the north over those from the east. The probabilities \(\{\pi_{ns}, \pi_{nw}, \pi_{es}, \pi_{ew}\}\) describe the likelihood of a car from the north or east proceeding south or west through the intersection. In our framework these are preselection rules. The time to pass through the crossing is denoted \(\tau_c\) while \(\tau_s\) and \(\tau_w\) represent the times to circle the north-south and east-west roads, respectively. The resulting Petri net is presented in~\Cref{fig:junction}.

\begin{figure}[h]
  \centering
  \def\tkzscl{.25}

  \definecolor{colorprio}{rgb}{1,0,0}
  \definecolor{colorgray}{gray}{0.6}

  \tikzset{place/.style={draw,circle,inner sep=2.5pt,semithick}}
  \tikzset{transition/.style={rectangle, thick,fill=black, minimum width=2mm, inner ysep=0.5pt}}
  \tikzset{token/.style={draw,circle,fill=black!80,inner sep=.35pt}}
  \tikzset{pre/.style={=stealth'}}
  \tikzset{post/.style={->,shorten >=1pt,>=stealth'}}
  \tikzset{-|/.style={to path={-| (\tikztotarget)}}, |-/.style={to path={|- (\tikztotarget)}}}
  \tikzset{Farhi/.style 2 args={dashed,dash pattern=on 1pt off 1pt,#1, postaction={draw,dashed,dash pattern=on 1pt off 1pt,#2,dash phase=1pt}}}
  \tikzset{arrowPetri/.style={>=latex,rounded corners=5pt,semithick}}

  \begin{tikzpicture}[scale=\tkzscl,font=\scriptsize]

    \def\p{2.2}

    \begin{scope}[shift={(0,0)}]
      \node[place]      (p_crossing_availability) at ($(0,0)$) {};
      \node (txt_a) at ($(p_crossing_availability) + (0.7*\p, 0)$) {\(A\)};

      \node[transition] (q_north) at ($(p_crossing_availability) + (-2*\p,-1.5*\p)$) {};
      \node (txt_zn) at ($(q_north) + (-0.8*\p, 0)$) {\(z_n\)};
      \node[transition] (q_east) at ($(p_crossing_availability) + (2*\p,-1.5*\p)$) {};
      \node (txt_ze) at ($(q_east) + (0.7*\p, 0)$) {\(z_e\)};

      \draw[->>,arrowPetri,colorprio]    (p_crossing_availability) |- ($(p_crossing_availability)+(-.5*\p,-.5*\p)$) -| (q_north);
      \draw[->,arrowPetri,colorprio]    (p_crossing_availability) |- ($(p_crossing_availability)+(.5*\p,-.5*\p)$) -| (q_east);

      \node[place]      (p_north) at ($(q_north)+(0, -1*\p)$) {};
      \node (txt_taucn) at ($(p_north) + (0.7*\p, 0)$) {\(\tau_c\)};
      \node[place]      (p_east) at ($(q_east)+(0, -1*\p)$) {};
      \node (txt_tauce) at ($(p_east) + (-0.8*\p, 0)$) {\(\tau_c\)};

      \draw[->,arrowPetri] (q_north) -- (p_north);
      \draw[->,arrowPetri] (q_east) -- (p_east);
      
      \node[transition]      (q_northsouth) at ($(p_north)+(-\p, -1.5*\p)$) {};
      \node (txt_pi_ns) at ($(q_northsouth) + (0.5*\p, 0.2*\p)$) {$\pi_{\text{ns}}$};
      \node[transition]      (q_northwest) at ($(p_north)+(\p, -1.5*\p)$) {};
      \node (txt_pi_nw) at ($(q_northwest) + (0.5*\p, 0.2*\p)$) {$\pi_{\text{nw}}$};

      \node[transition]      (q_eastsouth) at ($(p_east)+(-\p, -1.5*\p)$) {};
      \node (txt_pi_es) at ($(q_eastsouth) + (0.5*\p, 0.2*\p)$) {$\pi_{\text{es}}$};
      \node[transition]      (q_eastwest) at ($(p_east)+(\p, -1.5*\p)$) {};
      \node (txt_pi_ew) at ($(q_eastwest) + (-0.5*\p, 0.2*\p)$) {$\pi_{\text{ew}}$};

      \draw[->,arrowPetri]    (p_north) |- ($(p_north)+(-.5*\p,-.5*\p)$) -| (q_northsouth);
      \draw[->,arrowPetri]    (p_north) |- ($(p_north)+(.5*\p,-.5*\p)$) -| (q_northwest);

      \draw[->,arrowPetri]    (p_east) |- ($(p_east)+(-.5*\p,-.5*\p)$) -| (q_eastsouth);
      \draw[->,arrowPetri]    (p_east) |- ($(p_east)+(.5*\p,-.5*\p)$) -| (q_eastwest);

      \node[place] (p_south) at ($(p_north) + (0, -4*\p)$) {};
      \node (txt_taus) at ($(p_south) + (-0.8*\p, 0)$) {\(\tau_s\)};
      \node[place] (p_west) at ($(p_east) + (0, -4*\p)$) {};
      \node (txt_tauw) at ($(p_west) + (0.7*\p, 0)$) {\(\tau_w\)};

      \draw[->,arrowPetri]    (q_northsouth) |- ($(q_northsouth)+(.5*\p,-1.5*\p)$) -| (p_south);
      \draw[->,arrowPetri]    (q_northwest) |- ($(q_northwest)+(.5*\p,-0.75*\p)$) -| (p_west);

      \draw[->,arrowPetri,dashed,colorgray]    (q_northsouth) |- ($(q_northsouth)+(-1*\p,-.5*\p)$) -- ($(p_crossing_availability) + (-4*\p,\p)$) -| (p_crossing_availability);
      \draw[->,arrowPetri,dashed,colorgray]    (q_northwest) |- ($(q_northsouth)+(-1*\p,-.5*\p)$) -- ($(p_crossing_availability) + (-4*\p,\p)$) -| (p_crossing_availability);

      \draw[->,arrowPetri]    (q_eastwest) |- ($(q_eastwest)+(-.5*\p,-1.5*\p)$) -| (p_west);
      \draw[->,arrowPetri]    (q_eastsouth) |- ($(q_eastsouth)+(-.5*\p,-1.25*\p)$) -| (p_south);

      \draw[->,arrowPetri,dashed,colorgray]    (q_eastwest) |- ($(q_eastwest)+(\p,-.5*\p)$) -- ($(p_crossing_availability) + (4*\p,\p)$) -| (p_crossing_availability);
      \draw[->,arrowPetri,dashed,colorgray]    (q_eastsouth) |- ($(q_eastwest)+(\p,-.5*\p)$) -- ($(p_crossing_availability) + (4*\p,\p)$) -| (p_crossing_availability);

      \draw[->,arrowPetri] (p_west) |- ($(p_west) + (1.5*\p,-.5*\p)$) -- ($(q_east) + (1.5*\p,.5*\p)$) -- ($(q_east) + (.5*\p,.5*\p)$) -- (q_east);
      \draw[->,arrowPetri] (p_south) |- ($(p_south) + (-1.5*\p,-.5*\p)$) -- ($(q_north) + (-1.5*\p,.5*\p)$) -- ($(q_north) + (-.5*\p,.5*\p)$) -- (q_north);

    \end{scope}
    
  \end{tikzpicture}
  \caption{Two circular roads with a crossing \cite{Farhi2011}.}
  \label{fig:junction}
\end{figure}
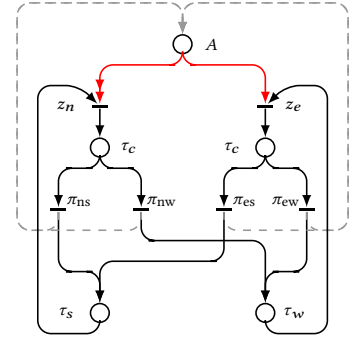
The reduced dynamics of this system is given by
\begin{subequations}
  \begin{align}
  z_n(t) & = &&\min\bigl\{A + z_n(t-\tau_c) + z_e(t-\tau_c) - z_e(\tau-\epsilon), \nonumber \\
         &   &&\quad  \pi_{\textrm{ns}}z_n(t-\tau_c-\tau_s) + \pi_{\textrm{es}}z_e(t-\tau_c-\tau_s)\bigr\} \,, \\
  z_e(t) & = &&\min\bigl\{A + z_n(t-\tau_c) + z_e(t-\tau_c) - z_n(\tau), \notag \\
         &   &&\quad  \pi_{\textrm{nw}}z_n(t-\tau_c-\tau_w) + \pi_{\textrm{ew}}z_e(t-\tau_c-\tau_w)\bigr\} \,.
  \end{align}
  \label{eq:junction}
\end{subequations}
By the same methods as in the PFAU example, it can be checked that this system fulfills~\Cref{as:1}~\eqref{as:1-1} and thus admits an invariant half-line.

\subsection{Emergency Call Center with Monitored Reservoir (EMS-B)~\cite[Chap. 5, Sec. 3]{boyet2022}}
We look at an extension of the emergency call center where a pool of ``reservoir assistant'' is introduced to specifically handle the briefing of the second level responder in the three-way call. This way, the first level responders are not congested by the availability or not of second level responders. The exact formulation used here is borrowed from~\cite[Chap. 5, Sec. 3]{boyet2022} where it is also thoroughly detailed.
\begin{figure}[h]
  \centering
  \def\tkzscl{.25}

  \definecolor{colorARM}{rgb}{0,0,1}
  \definecolor{colorARMres}{rgb}{.92,.5,.11}
  \definecolor{colorAMU}{rgb}{1,0,0}
  \definecolor{colorexit}{rgb}{0.4,0.4,0.4}
  
  \tikzset{place/.style={draw,circle,inner sep=2.5pt,semithick}}
  \tikzset{transition/.style={rectangle, thick,fill=black, minimum width=2mm,inner ysep=0.5pt}}
  \tikzset{token/.style={draw,circle,fill=black!80,inner sep=.35pt}}
  \tikzset{pre/.style={=stealth'}}
  \tikzset{post/.style={->,shorten >=1pt,>=stealth'}}
  \tikzset{-|/.style={to path={-| (\tikztotarget)}}, |-/.style={to path={|- (\tikztotarget)}}}
  \tikzset{Farhi/.style 2 args={dashed,dash pattern=on 1pt off 1pt,#1, postaction={draw,dashed,dash pattern=on 1pt off 1pt,#2,dash phase=1pt}}}
  \tikzset{arrowPetri/.style={>=latex,rounded corners=5pt,semithick}}
  
  \begin{tikzpicture}[scale=\tkzscl,font=\scriptsize]
  
  \def\p{2.2}
  
  \begin{scope}[shift={(0,0)}]
  \node[place] (pool_arm) at ($(-2*\p,-1.5*\p)$) {};
  \node (txt_Na) at ($(pool_arm)+(-1.3,0)$) {$N_A$};
  
  \node[transition]    (q_arrivals)                          at    (0, 0) {};
  \node[place]         (p_inc_calls)                         at    ($(q_arrivals) + (0, \p)$)  {};
  \node[transition]    (q_inc_calls)                         at    ($(p_inc_calls) + (0, \p)$)  {};
  
  \node[place]         (p_arrivals)        at    ($(q_arrivals) + (0,-\p)$) {};
  \node[transition]    (q_debut_NFU)       at    ($(p_arrivals) + (0,-\p)$) {};
  \node[transition]    (q_synchro)         at    ($(p_arrivals) + (2*\p,-\p)$) {};
  
  \node[place]         (pool_res)          at    ($(q_synchro) + (2*\p,.5*\p)$) {};
  \node (txt_Nr) at ($(pool_res)+(1.1,.7)$) {$N_R$};
  
  \draw[->,arrowPetri,colorARM] (p_arrivals)  |- ($(p_arrivals)+(.5*\p,-.5*\p)$) -| (q_synchro);
  \draw[->,arrowPetri,colorARM] (q_debut_NFU) |- ($(q_debut_NFU)+(0,-.5*\p)$) -| (pool_arm);
  \draw[->,arrowPetri,colorARM] (p_arrivals)  -- (q_debut_NFU);
  \draw[->,arrowPetri,colorARM] (q_arrivals) -- (p_arrivals);
  
  \node[place]         (p_synchro)           at    ($(q_synchro)  + (0,-\p)$) {};
  \node[transition]    (q_unsynchro)         at    ($(p_synchro)  + (0,-\p)$) {};
  
  \node[place]      (pool_amu)    at    ($(pool_res)+(6*\p,0)$) {};
  \node (txt_Nm) at ($(pool_amu)+(-1.3,0)$) {$N_P$};
  
  \node[place]         (p_synchro2)          at    ($(pool_amu)  + (-4*\p,.5*\p)$) {};
  \node[transition]    (q_unsynchro2)        at    ($(p_synchro2)  + (0,-\p)$) {};
  \node[transition]    (q_debut_AMU)         at    ($(p_synchro2) + (0,\p)$) {};
  \node[place]         (p_waiting)           at    ($(q_debut_AMU)  + (2*\p,1.5*\p)$) {};
  \node[place]         (p_consult_AMU)       at    ($(q_unsynchro2)+ (0,-1*\p)$) {};
  \node[transition]    (q_end_consult_AMU)   at    ($(p_consult_AMU)+(0,-\p)$) {};
  
  \node[place]         (p_synchro3)          at    ($(p_synchro2)       + (2*\p,0)$) {};
  \node[transition]    (q_unsynchro3)        at    ($(q_unsynchro2)     + (2*\p,0)$) {};
  \node[transition]    (q_debut_AMU_2)       at    ($(q_debut_AMU)      + (2*\p,0)$) {};
  \node[transition]    (q_end_consult_AMU_2) at    ($(q_end_consult_AMU)+ (2*\p,0)$) {};
  \node[place]         (p_consult_AMU_2)     at    ($(p_consult_AMU)    + (2*\p,0)$) {};
  
  \draw[arrowPetri]    (q_inc_calls) |- ($(q_inc_calls)+(-.35*\p,.25*\p)$);
  \draw[dashed, arrowPetri]    (q_inc_calls) |- ($(q_inc_calls)+(-1.5*\p,.25*\p)$);
  \draw[->,arrowPetri] (q_inc_calls) -- (p_inc_calls);
  \draw[->,arrowPetri]                                    (p_inc_calls) -- (q_arrivals);
  \draw[->,arrowPetri,colorARM]                           (q_unsynchro) -- ($(q_unsynchro)+(0,-.5*\p)$)  -|  (pool_arm);
  
  \draw[->,arrowPetri,Farhi={colorARM}{colorARMres}]    (q_synchro) -- (p_synchro);
  \draw[->,arrowPetri,Farhi={colorARM}{colorARMres}]    (p_synchro) -- (q_unsynchro);
  
  \draw[->,arrowPetri,colorARMres]    (pool_res) -- ($(q_synchro)+(.5*\p,.5*\p)$) -- (q_synchro);
  \draw[->,arrowPetri,colorARMres]    (q_unsynchro) |- ($(q_unsynchro)+(.5*\p,-.5*\p)$)-| (pool_res);
  
  \draw[->>>,arrowPetri,colorARMres]   (pool_res) |- ($(q_debut_AMU)+(-.75*\p,.75*\p)$) -- (q_debut_AMU);
  \draw[->,arrowPetri,colorARMres]    (q_unsynchro2) |- ($(q_unsynchro2)+(-\p,-.5*\p)$) -- (pool_res);
  
  \draw[->>,arrowPetri,colorARMres]  (pool_res) |- ($(q_debut_AMU_2)+(-.75*\p,.75*\p)$) -- (q_debut_AMU_2);
  \draw[->,arrowPetri,colorARMres]    (q_unsynchro3) |- ($(q_unsynchro2)+(-\p,-.5*\p)$) -- (pool_res);
  
  \draw[->,arrowPetri,Farhi={colorAMU}{colorARMres}]    (q_debut_AMU) -- (p_synchro2);
  \draw[->,arrowPetri,Farhi={colorAMU}{colorARMres}]    (p_synchro2) -- (q_unsynchro2);
  \draw[->,arrowPetri,Farhi={colorAMU}{colorARMres}]    (q_debut_AMU_2) -- (p_synchro3);
  \draw[->,arrowPetri,Farhi={colorAMU}{colorARMres}]    (p_synchro3) -- (q_unsynchro3);
  
  \draw[->,arrowPetri]    (p_waiting) |- ($(p_waiting)+(-.5*\p,-.5*\p)$) -| (q_debut_AMU);
  \draw[->,arrowPetri]    (p_waiting) -- (q_debut_AMU_2);
  \draw[->,arrowPetri]    (q_unsynchro) |- ($(q_unsynchro)+(1.5*\p,-1*\p)$) -| ($(pool_amu)+(.5*\p,0)$) |- ($(p_waiting)+(0,.75*\p)$) -- (p_waiting);
  
  \draw[->,arrowPetri,colorARM]                           (pool_arm)    |- ($(q_arrivals)+(-.5*\p,.5*\p)$) -- (q_arrivals);
  
  \node (txt_taus) at ($(p_synchro)+(.5*\p,0)$) {$\tau_2$};
  \node (txt_taus) at ($(p_synchro2)+(.5*\p,0)$) {$\tau_2$};
  \node (txt_taus) at ($(p_synchro3)+(.5*\p,0)$) {$\tau_2$};
  \node (txt_tau1) at ($(p_arrivals.center)+(.5*\p,0)$) {${\tau}_1$};
  \node (txt_tau3) at ($(p_consult_AMU.center)+(.5*\p,0)$) {${\tau}_3$};
  \node (txt_tau3) at ($(p_consult_AMU_2.center)+(.5*\p,0)$) {${\tau}_3$};
  \node (txt_pi)  at ($(p_arrivals.center)+(-1.4,-1.65)$) {$1\!-\!\pi$};
  \node (txt_pi3) at ($(p_arrivals.center)+(3.7, -1.65)$){$\pi$};
  \node (txt_alpha3) at ($(p_waiting.center)+(-4.8,-1)$){$\alpha$};
  \node (txt_alpha)  at ($(p_waiting.center) +(1.4,-1)$) {$1\!-\!\alpha$};
  
  \draw[->,arrowPetri,colorAMU]    (q_unsynchro2)      -- (p_consult_AMU);
  \draw[->,arrowPetri,colorAMU]    (p_consult_AMU)     -- (q_end_consult_AMU);
  \draw[->,arrowPetri,colorAMU]    (q_end_consult_AMU) -- ($(q_end_consult_AMU)+(0,-1)$) -| (pool_amu);
  \draw[->>,arrowPetri,colorAMU]   (pool_amu)      |- ($(q_debut_AMU)+(.5*\p,.5*\p)$)      -- (q_debut_AMU);
  
  \draw[->,arrowPetri,colorAMU]    (q_unsynchro3)       -- (p_consult_AMU_2);
  \draw[->,arrowPetri,colorAMU]    (p_consult_AMU_2)     -- (q_end_consult_AMU_2);
  \draw[->,arrowPetri,colorAMU]    (q_end_consult_AMU_2) -- ($(q_end_consult_AMU_2)+(0,-1)$) -| (pool_amu);
  \draw[->,arrowPetri,colorAMU]    (pool_amu)      |- ($(q_debut_AMU_2)+(.5*\p,.5*\p)$)      -- (q_debut_AMU_2);
  
  \node (txt_z0) at ($(q_inc_calls)+(1*\p,0)$) {$z_0 = \lambda t$};
  \node (txt_z1) at ($(q_arrivals)+(.5*\p,0)$) {$z_1$};
  \node (txt_z2) at ($(q_debut_NFU)+(.5*\p,0)$) {$z_2$};
  \node (txt_z3) at ($(q_synchro)+(.5*\p,0)$) {$z_3$};
  \node (txt_z4) at ($(q_unsynchro)+(.5*\p,0)$) {$z_4$};
  \node (txt_z5) at ($(q_debut_AMU)+(.5*\p,0)$) {$z_5$};
  \node (txt_z6) at ($(q_unsynchro2)+(.5*\p,0)$) {$z_6$};
  \node (txt_z7) at ($(q_end_consult_AMU)+(.5*\p,0)$) {$z_7$};
  \node (txt_z5) at ($(q_debut_AMU_2)+(.5*\p,0)$) {$z_5'$};
  \node (txt_z6) at ($(q_unsynchro3)+(.5*\p,0)$) {$z_6'$};
  \node (txt_z7) at ($(q_end_consult_AMU_2)+(.5*\p,0)$) {$z_7'$};
  
  \end{scope}
  
  \end{tikzpicture}	
  
  \caption{Medical emergency call center with a monitored reservoir (EMS-B)~\cite{boyet2021}.}
  \label{fig:SAMU2}
  \end{figure}
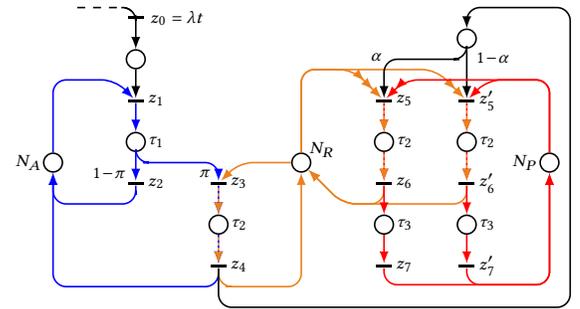

  This example illustrates a more complex Petri net with priorities and demonstrates the potential nonconcave behavior of the throughput in our dynamics. Specifically, adding more second level responders may {\em slow down\/} the treatment of very urgent calls due to occupying reservoir assistants with the urgent calls, effectively disallowing them from being briefing these very urgent calls until they are done with the urgent briefings.

  Despite this, the system can be proved to fulfill~\Cref{as:1}~\eqref{as:1-1} and thus having an invariant half-line. The dynamics and half-line solution can be found in~\cite[Chap. 5, Sec. 3]{boyet2022}.

\if{\subsection{The ASGARD center~\cite{fafesoali}}
This example is lifted from the end-of-studies project~\cite{fafesoali}. 
Clarify the model. Say that ASGARD is a technician center of ENEDIS (subsidiary company of EDF
in charge of Electricity distribution in France).

\begin{figure}[h]
  \centering
  \def\tkzscl{.29}
  \vspace{-.2cm}

  \definecolor{colorARM}{rgb}{0,0,1}
  \definecolor{colorARMres}{rgb}{0.92,0.5,0.11}
  \definecolor{colorAMU}{rgb}{1,0,0}
  \definecolor{colorexit}{rgb}{0.4,0.4,0.4}
  
  \tikzset{place/.style={draw,circle,inner sep=2.5pt,semithick}}
  \tikzset{transition/.style={rectangle, thick,fill=black, minimum width=2mm,inner ysep=0.5pt, minimum width=2mm}}
  \tikzset{jeton/.style={draw,circle,fill=black!80,inner sep=.35pt}}
  \tikzset{pre/.style={=stealth'}}
  \tikzset{post/.style={->,shorten >=1pt,>=stealth'}}
  \tikzset{-|/.style={to path={-| (\tikztotarget)}}, |-/.style={to path={|- (\tikztotarget)}}}
  \tikzset{Farhi/.style 2 args={dashed,dash pattern=on 1pt off 1pt,#1, postaction={draw,dashed,dash pattern=on 1pt off 1pt,#2,dash phase=1pt}}}
  \tikzset{arrowPetri/.style={>=latex,rounded corners=5pt,semithick}}
  
  \begin{tikzpicture}[scale=\tkzscl,font=\scriptsize]
  
  \def\p{2.4}
  
  \begin{scope}[shift={(0,4)}]
  \node[place] (pool_supervisors) at ($(-2*\p,-1.5*\p)$) {};
  \node (txt_Ns) at ($(pool_supervisors)+(-1.3,0)$) {$N_S$};
  
  \node[transition]    (q_incoming)                          at    (0, 0) {};
  \node[place]         (p_inc_calls)                         at    ($(q_incoming) + (0, \p)$)  {};
  \node[transition]    (q_inc_calls)                         at    ($(p_inc_calls) + (0, \p)$)  {};
  \node[place]         (p_arrivals)       at    ($(q_incoming) + (0,-\p)$) {};
  \node[transition]    (q_begin_U)        at    ($(p_arrivals) + (0,-1.5*\p)$) {};
  \node[transition]    (q_begin_VU)       at    ($(p_arrivals) + (2*\p,-1.5*\p)$) {};
  
  \draw[->,arrowPetri,Farhi={colorARM}{black}] (p_arrivals)  -- (q_begin_U);
  \draw[->,arrowPetri,Farhi={colorARM}{black}] (q_incoming) -- (p_arrivals);
  \draw[->,arrowPetri,Farhi={colorARM}{black}] (p_arrivals)  |- ($(p_arrivals)+(.5*\p,-.5*\p)$) -| (q_begin_VU);

  \node[place]         (p_U)           at    ($(q_begin_U)  + (0,-\p)$) {};
  \node[transition]    (q_U)         at    ($(p_U)  + (0,-\p)$) {};
  \node[place]         (p_VU)           at    ($(q_begin_VU)  + (0,-\p)$) {};
  \node[transition]    (q_VU)         at    ($(p_VU)  + (0,-\p)$) {};

  \node[place]         (p_U1)           at    ($(q_U)  + (0,-\p)$) {};
  \node[transition]    (q_U1)         at    ($(p_U1)  + (0,-\p)$) {};
  \node[place]         (p_VU1)           at    ($(q_VU)  + (0,-\p)$) {};
  \node[transition]    (q_VU1)         at    ($(p_VU1)  + (0,-\p)$) {};

  \node[place]         (p_U2)           at    ($(q_U1)  + (0,-\p)$) {};
  \node[transition]    (q_U2)         at    ($(p_U2)  + (0,-\p)$) {};
  \node[place]         (p_VU2)           at    ($(q_VU1)  + (0,-\p)$) {};
  \node[transition]    (q_VU2)         at    ($(p_VU2)  + (0,-\p)$) {};

  \node[place]         (p_U3)           at    ($(q_U2)  + (0,-\p)$) {};
  \node[transition]    (q_U3)         at    ($(p_U3)  + (0,-\p)$) {};
  \node[place]         (p_VU3)           at    ($(q_VU2)  + (0,-\p)$) {};
  \node[transition]    (q_VU3)         at    ($(p_VU3)  + (0,-\p)$) {};

  \draw[->,arrowPetri,colorARM] (pool_supervisors) |- ($(q_incoming)+(-1*\p,.5*\p)$) -- ($(q_U)+(-1*\p,.5*\p)$)  -- ($(q_U)+(-0.5*\p,.5*\p)$) -- (q_U);
  \draw[->,arrowPetri,colorARM] (pool_supervisors) |- ($(q_incoming)+(-1*\p,.5*\p)$) -- ($(q_U)+(-1*\p,.5*\p)$)  -- ($(q_VU)+(-0.5*\p,.5*\p)$) -- (q_VU);
  \draw[->,arrowPetri,colorARM] (pool_supervisors) |- ($(q_incoming)+(-1*\p,.5*\p)$) -- ($(q_U2)+(-1*\p,.5*\p)$)  -- ($(q_U2)+(-0.5*\p,.5*\p)$) -- (q_U2);
  \draw[->,arrowPetri,colorARM] (pool_supervisors) |- ($(q_incoming)+(-1*\p,.5*\p)$) -- ($(q_U2)+(-1*\p,.5*\p)$)  -- ($(q_VU2)+(-0.5*\p,.5*\p)$) -- (q_VU2);

  \draw[arrowPetri]    (q_inc_calls) |- ($(q_inc_calls)+(-.35*\p,.25*\p)$);
  \draw[dashed, arrowPetri]    (q_inc_calls) |- ($(q_inc_calls)+(-1.5*\p,.25*\p)$);
  \draw[->,arrowPetri] (q_inc_calls) -- (p_inc_calls);
  \draw[->,arrowPetri] (p_inc_calls) -- (q_incoming);

  \draw[->,arrowPetri, colorARM] (q_begin_U) |- ($(q_begin_U)+(0.5*\p,-0.5*\p)$) -- ($(q_begin_VU)+(0.5*\p,-0.5*\p)$) -- ($(q_VU3)+(0.5*\p,-0.5*\p)$) -| (pool_supervisors);
  \draw[->,arrowPetri, colorARM] (q_begin_VU) |- ($(q_begin_VU)+(0.5*\p,-0.5*\p)$) -- ($(q_VU3)+(0.5*\p,-0.5*\p)$) -| (pool_supervisors);
  \draw[->,arrowPetri, colorARM] (q_U1) |- ($(q_U1)+(0.5*\p,-0.5*\p)$) -- ($(q_VU1)+(0.5*\p,-0.5*\p)$) -- ($(q_VU3)+(0.5*\p,-0.5*\p)$) -| (pool_supervisors);
  \draw[->,arrowPetri, colorARM] (q_VU1) |- ($(q_VU1)+(0.5*\p,-0.5*\p)$) -- ($(q_VU3)+(0.5*\p,-0.5*\p)$) -| (pool_supervisors);

  \draw[->,arrowPetri,colorARM] (q_U3) -- ($(q_U3)+(0,-.5*\p)$)  -|  (pool_supervisors);
  \draw[->,arrowPetri,colorARM] (q_VU3) -- ($(q_VU3)+(0,-.5*\p)$)  -|  (pool_supervisors);
  
  \draw[->,arrowPetri]    (q_begin_U) -- (p_U);
  \draw[->,arrowPetri]    (p_U) -- (q_U);
  \draw[->,arrowPetri]    (q_begin_VU) -- (p_VU);
  \draw[->,arrowPetri]    (p_VU) -- (q_VU);

  \draw[->,arrowPetri,Farhi={colorARM}{black}]    (q_U) -- (p_U1);
  \draw[->,arrowPetri,Farhi={colorARM}{black}]    (p_U1) -- (q_U1);
  \draw[->,arrowPetri,Farhi={colorARM}{black}]    (q_VU) -- (p_VU1);
  \draw[->,arrowPetri,Farhi={colorARM}{black}]    (p_VU1) -- (q_VU1);

  \draw[->,arrowPetri]    (q_U1) -- (p_U2);
  \draw[->,arrowPetri]    (p_U2) -- (q_U2);
  \draw[->,arrowPetri]    (q_VU1) -- (p_VU2);
  \draw[->,arrowPetri]    (p_VU2) -- (q_VU2);

  \draw[->,arrowPetri,Farhi={colorARM}{black}]    (q_U2) -- (p_U3);
  \draw[->,arrowPetri,Farhi={colorARM}{black}]    (p_U3) -- (q_U3);
  \draw[->,arrowPetri,Farhi={colorARM}{black}]    (q_VU2) -- (p_VU3);
  \draw[->,arrowPetri,Farhi={colorARM}{black}]    (p_VU3) -- (q_VU3);

  \draw[->,arrowPetri,colorARM]   (pool_supervisors)    |- ($(q_incoming)+(-.5*\p,.5*\p)$) --(q_incoming);
  
  \node (txt_pi1) at ($(p_arrivals.west) + (0.1*\p,-0.5*\p)$) [left] {$\pi_1$};
  \node (txt_pi2) at ($(p_arrivals.east) + (0,-0.1*\p)$) [right] {$\pi_2$};
  
  \node (txt_z0) at ($(q_inc_calls)+(1*\p,0)$) {$z_0 = \lambda t$};
  \node (txt_z1) at ($(q_incoming)+(.5*\p,0)$) {$z_1$};
  \node (txt_z2) at ($(q_U)+(-0.5*\p,0)$) {$z_2$};
  \node (txt_z2p) at ($(q_VU)+(-0.5*\p,0)$) {$z_2^{\prime}$};
  \node (txt_z3) at ($(q_U2)+(-0.5*\p,0)$) {$z_3$};
  \node (txt_z3p) at ($(q_VU2)+(-0.5*\p,0)$) {$z_3^\prime$};
  \end{scope}
  \end{tikzpicture}	
  \caption{Asgard service call center~\cite{fafesoali}}
  \label{fig:asgard}
\end{figure}
}\fi
\section*{Concluding remarks}
We gave sufficient conditions which entail that a piecewise linear dynamics
admits an invariant half-line. This provides a partial answer to an open question
which arose in a series of works~\cite{maxplusblondel,Farhi2011,emergency15,boyet2021,boyet2022}. This provides an alternative to earlier approaches, relying
on exhaustive policy enumeration and polyhedral computations. All our assumptions, but one,
are automatically satisfied for the class of examples of call centers
and of road networks under consideration. Our last assumption turns out
to be satisfied in all our examples, but its verification remains non trivial.
We conjecture that it could be implied by appropriate topological conditions
on the Petri net.  We leave the investigation of this question
for further work.
\section*{Acknowledgement}
The authors were partially supported by the URGE project of the Bernoulli Lab,
joint between AP-HP \& INRIA. 
 \bibliographystyle{alpha}
\bibliography{note}
\appendix
\section{Proof of Lemma~\ref{lemma-lexicographic}}
\label{append-lemma-lexicographic}
  We begin by showing that the existence of an invariant half line of~\Cref{eq:tds} implies that the lexicographic system is satisfied.
  Substituting in the half line solution into~\Cref{eq:tds} gives
  \begin{align}
    u_i + t \rho_i = \min_{a \in \actions_i} \biggl( r_i^a + \sum_\tau {[P_{\tau}^a]}_i (u + (t-\tau) \rho) \biggr) \,, \label{eq:halfline}
  \end{align}
  for \(t \geq t_1\).
  To show that~\eqref{eq:lexeta} is satisfied we divide by \(t\) on both sides to get
  \begin{align*}
    \frac{u_i}{t} + \rho_i = \min_{a \in \actions_i} \biggl( \frac{r_i^a + \sum_\tau {[P_{\tau}^a]}_i (u - \tau \rho)}{t} + \sum_\tau {[P_{\tau}^a]}_i \rho  \biggr) \,,
  \end{align*}
  where taking the limit \(t \to \infty\) yields
  \begin{align*}
    \rho_i = \min_{a \in \actions_i} {[P^a(1)]}_i \rho \,,
  \end{align*} 
  by continuity.

  To show~\eqref{eq:lexu} we rewrite condition~\eqref{eq:halfline} as
  \begin{align}
    u_i = \min_{a \in \actions_i} \biggl( r_i^a + \sum_\tau {[P_{\tau}^a]}_i u - \sum_\tau {[P_{\tau}^a]}_i \tau \rho + t\biggl( \sum_\tau {[P_{\tau}^a]}_i \rho - \rho_i \biggr) \biggr) \,. \label{eq:halfline-u}
  \end{align}
  If \(a \not\in \actions_i^\star\) the factor of \(t\), \(\sum_\tau {[P_{\tau}^a]}_i \rho - \rho_i \) is strictly positive and so the terms inside the minimum tends to \(+\infty\) as \(t \to \infty\).
  We deduce that for \(t\) large enough the minimum is achieved only for \(a \in \actions_i^\star\) and so with the identification~\eqref{eq:lexeta} we get
  \begin{align*}
    u_i & = \min_{a \in \actions_i^\star} \biggl( r_i^a + \sum_\tau {[P_{\tau}^a]}_i u - \sum_\tau {[P_{\tau}^a]}_i \tau \rho + t\left(\rho_i - \rho_i \right) \biggr) \\
    & =\min_{a \in \actions_i^\star} \biggl( r_i^a + \sum_\tau {[P_{\tau}^a]}_i u - \sum_\tau {[P_{\tau}^a]}_i \tau \rho \biggr) \\
    & =\min_{a \in \actions_i^\star} \biggl( r_i^a + \sum_\tau {[P_{\tau}^a]}_i u - {[P_{\tau}^a]}_i \rho - \sum_\tau {[P_{\tau}^a]}_i (\tau-1) \rho \biggr) \\ 
    & =\min_{a \in \actions_i^\star} \biggl( r_i^a + \sum_\tau {[P_{\tau}^a]}_i u - \rho_i - {[S^a]}_i \rho \biggr) \,,
  \end{align*}
  which shows the implication in one direction.

  Conversely, suppose that~\eqref{eq:lexeta} and~\eqref{eq:lexu} hold. It suffices to check that~\eqref{eq:halfline-u} is satisfied for large enough \(t\).
  By definition of \(\actions_i^\star\), every term inside the minimum arising from \(a \in \actions_i \setminus \actions_i^\star\) tends to \(+\infty\) when \(t \to \infty\). Thus, there exists some \(t_1\) such that for all \(t \geq t_1\) the expression
  \begin{align*}
    \min_{a \in \actions_i} \biggl( r_i^a + \sum_\tau {[P_{\tau}^a]}_i u - \sum_\tau {[P_{\tau}^a]}_i \tau \rho + t \biggl( \sum_\tau {[P_{\tau}^a]}_i \rho - \rho_i \biggr) \biggr)
  \end{align*}
  reduces to
  \begin{align*}
    \min_{a \in \actions_i^\star} \biggl( r_i^a + \sum_\tau {[P_{\tau}^a]}_i u - \sum_\tau {[P_{\tau}^a]}_i \tau \rho \biggr) \,,
  \end{align*}
  which by~\eqref{eq:lexu} coincides with \(u_i\) recovering~\eqref{eq:halfline-u}.
  \hfill\qed

  \section{Proof of Lemma~\ref{lemma:simple}}\label{append:lemma:simple}
The matrix    \(P\) being nonexpansive specifically implies that \(P\) is nonexpansive
  with respect to the weighted sup-norm (by the equivalence of norms on \(\R^n\))
  \[
  \|x\|_{e,\infty}\coloneqq \max_i |\frac{x_i}{e_i}| \enspace,
  \]
  meaning that
  \[
  \|Px\|_{e\infty}\leq \|x\|_{e,\infty} \,.
  \]
  Consider now the equation
  \[
  y= \alpha P y + r
  \]
  where \(r\in \R^n\). We have
  \[
  u= (I-\alpha P)^{-1} r
  \]
  and we can deduce (easy induction: \(y= \lim y_k\), \(y_k = r +\alpha P r + \alpha^2 P^r + \dots \) up to rank \(k-1\), and
  control the norm),
  show that
  \[
  \|u\|_{e,\infty} \leq \frac{\|r\|_{e,\infty}}{1-\alpha}
  \]
  It follows that the resolvent \((I-\alpha P)^{-1}\) cannot have a pole of order
  greater than \(1\) at \(\alpha=1\), and so, by~\cite{kato}, \(1\) is semisimple.
\hfill\qed 

\end{document}